\documentclass[a4paper]{amsart}

\usepackage[utf8]{inputenc}
\usepackage[T1]{fontenc}

\usepackage{amsthm, amssymb, amsmath, amsfonts, mathrsfs}

\usepackage{mathtools}
\usepackage[ocgcolorlinks, colorlinks=true, pdfstartview=FitV, linkcolor=blue, citecolor=blue, urlcolor=blue, pagebackref=false]{hyperref}

\usepackage{microtype}

\newtheorem{thm}{Theorem}[section]
\newtheorem{prop}[thm]{Proposition}
\newtheorem{lem}[thm]{Lemma}

\theoremstyle{remark}
\newtheorem{rem}[thm]{Remark}

\def\les{\lesssim}

\renewcommand{\leq}{\leqslant}
\renewcommand{\geq}{\geqslant}

\newcommand{\B}{\mathbb{B}}

\newcommand{\cG}{\mathscr{G}}

\newcommand{\R}{\mathbb{R}}
\newcommand{\C}{\mathcal{C}}

\newcommand{\Z}{\mathbb{Z}}
\renewcommand{\P}{\mathbb{P}}

\newcommand{\ov}{\overline}
\newcommand{\un}{\underline}

\newcommand{\eps}{\varepsilon}
\def\d{{\mathrm{d}}}

\newcommand{\la}{\langle}
\newcommand{\lla}{\left\langle}
\newcommand{\ra}{\rangle}
\newcommand{\rra}{\right\rangle}

\newcommand{\V}{\mathcal{V}}

\newcommand{\G}{\mathcal{G}}

\numberwithin{equation}{section}

\title[high order correctors and two-scale expansions]{High order correctors and two-scale expansions in stochastic homogenization}
\author{Yu Gu}

\address[Yu Gu]{Department of Mathematics, Building 380, Stanford University, Stanford, CA, 94305, USA}

\begin{document}
\begin{abstract}

In this paper, we study high order correctors in stochastic homogenization. We consider elliptic equations in divergence form on $\Z^d$, with the random coefficients constructed from i.i.d. random variables. We prove moment bounds on the high order correctors and their gradients under dimensional constraints. It implies the existence of stationary correctors and stationary gradients in high dimensions. As an application, we prove a two-scale expansion of the solutions to the random PDE, which identifies the first and higher order random fluctuations in a strong sense.

\bigskip




\end{abstract}
\maketitle
%
%
%
%
%
%
%
%
%
\section{Main result}

Quantitative stochastic homogenization has witnessed important progress in recent years, and a major contribution of the groundbreaking work of Gloria-Otto is to prove the high order moment estimates on the corrector \cite{gloria2011optimal,gloria2014quantitative}.

The result in the discrete setting can be described as follows. Let $\B$ be the set of nearest neighbor edges in $\Z^d$, and $\{e_i,i=1,\ldots,d\}$ be the canonical basis of $\Z^d$. On a probability space $(\Omega,\mathcal{F},\P)$, we have a sequence of i.i.d. non-degenerate random conductances, denoted by $\{\omega_e\}_{e\in\B}$ with $\omega_e\in (\delta,1)$ for some $\delta>0$. We define $a:\Z^d\to \R^{d\times d}$ as a random diagonal matrix field such that 
\[
a(x)=\mathrm{diag}(a_1(x),\ldots,a_d(x))=\mathrm{diag}(\omega_{(x,x+e_1)},\ldots,\omega_{(x,x+e_d)}).
\] 
The \emph{regularized} corrector equation in the direction $\xi\in\R^d$ says
\begin{equation}
(\lambda+\nabla^* a(x)\nabla) \phi_\xi^\lambda(x)=-\nabla^* a(x)\xi, \  \ x\in \Z^d.
\label{eq:1stcor}
\end{equation}
Here the discrete gradient and divergence for $f:\Z^d\to \R$ and $F:\Z^d\to \R^d$ are defined as
\[
\nabla f=(\nabla_1f,\ldots,\nabla_d f), \ \ \nabla^*F=\sum_{i=1}^d \nabla_i^* F_i,
\]
with 
\[
\nabla_i f(x)=f(x+e_i)-f(x), \ \ \nabla_i^* F_i=F_i(x-e_i)-F_i(x).
\]
It was shown in \cite[Proposition 2.1]{gloria2011optimal} that when $d\geq 3$, for any $p\geq 1$, there exists a constant $C=C(d,\delta,p,\xi)>0$ such that  
\begin{equation}
\la |\phi_\xi^\lambda|^p\ra \leq C
\label{eq:gobound}
\end{equation}
uniformly in $\lambda>0$, where $\la\cdot \ra$ denotes the expectation on $(\Omega,\mathcal{F},\P)$. 
In particular, it implies the existence of a stationary corrector when $d\geq 3$: there exists a zero-mean stationary random field $\phi_\xi$ solving 
\begin{equation}
-\nabla^* a(x)\nabla \phi_\xi(x) =\nabla^* a(x)\xi.
\label{eq:1stcoreq}
\end{equation}

The first goal of this paper is to go beyond the first order correctors, 
and present a proof of \eqref{eq:gobound} for high order correctors. As an application of high order correctors, we identity the first and higher order fluctuations in stochastic homogenization in a strong sense that will be specified later.

\subsection{High order correctors}

We will stay in the same setting but further assume $\{\omega_e\}_{e\in\B}$ satisfies the log-Soblev inequality
\begin{equation}
\left\la \zeta^2 \log \frac{\zeta^2}{\la \zeta^2\ra}\right\ra \leq \frac{1}{\rho}\lla \sum_e |\partial_e \zeta|^2\rra
\label{eq:logsob}
\end{equation}
for some $\rho>0$. Here $\partial_e$ is the weak derivative with respect to $\omega_e$ and $\zeta:\Omega\to \R$ is any function so that the r.h.s. of \eqref{eq:logsob} makes sense. Through a formal two-scale expansion in Section~\ref{s:2scale}, we define the regularized $n-$th order corrector, and denote them by $\psi_n^\lambda$. The following is our first main result.

\begin{thm}
Fix any $n\geq 2$. When $d\geq 2n-1$, for any $p\geq 1$, there exists $C=C(d,\delta,\rho,p)>0$ such that
\[
\la |\nabla \psi_n^\lambda|^p \ra\leq C 
\]
uniformly in $\lambda>0$. When $d\geq 2n+1$, the same result holds for $\psi_n^\lambda$, i.e., 
\[
\la |\psi_n^\lambda|^p \ra\leq C
\]
uniformly in $\lambda>0$. In particular, for the $n-$th order corrector, it has a stationary gradient when $d\geq 2n-1$ and it is stationary when $d\geq 2n+1$.
\label{t:mainTh}
\end{thm}

\begin{rem}
The log-Soblev inequality \eqref{eq:logsob} holds for i.i.d. random variables with a continuous density, hence can not deal with distributions with atoms. A weaker version is presented in \cite[Equation (7)]{marahrens2013annealed}, covering all possible i.i.d. ensembles. We believe that the approach in this paper can be applied with the weaker version (with extra technicalities), and Theorem~\ref{t:mainTh} holds for any i.i.d. ensemble. To keep the presentation simple, we choose to work with \eqref{eq:logsob}.
\end{rem}

The study of stochastic homogenization started from the early work of Kozlov \cite{kozlov1979averaging} and Papanicolaou-Varadhan \cite{papanicolaou1979boundary}, and revived recently from various quantitative perspectives \cite{caffarelli2010rates,gloria2011optimal,gloria2012optimal,gloria2013fluctuation,gloria2013quantification,marahrens2013annealed,gloria2014optimal,gloria2014quantitative,gloria2014regularity,gloria2014improved,gloria2015optimal,mourrat2011variance,mourrat2012kantorovich,armstrong1,armstrong2,armstrong3,armstrong4}. While the first order correctors have been analyzed extensively due to their role in determining the effective coefficients and proving convergence in homogenization, the high order correctors have been receiving less attention. 
Our interest in the high order correctors comes from the comparison between a pointwise two-scale expansion and a large scale central limit theorem derived for the solutions to the random PDE
\begin{equation}
-\nabla\cdot a(\frac{x}{\eps})\nabla u_\eps(x)=f(x).
\label{eq:spde}
\end{equation}
The results in \cite{gu-mourrat,gu-mourratmms} showed that when $d\geq 3$, the first order corrector represents the local fluctuation, which is measured by $u_\eps(x)-\la u_\eps(x)\ra$ for fixed $x\in\R^d$, but does not suggest the global large scale fluctuation, which is measured weakly in space by $\int (u_\eps-\la u_\eps\ra) g$ with test function $g$. We expect the surprising phenomenon may be explained by high order correctors which only become visible in the weak sense due to strong correlations; see a discussion in Section~\ref{s:compa}. 

In Section~\ref{s:2scale}, we will construct high order correctors directly from a formal two-scale expansion, and there is a slightly different way of characterizing the high order correctors as the ``high order intrinsic polynomials'' that come out of the Liouville theorem: as the first order corrector $\phi_\xi$ is defined so that $\xi\cdot x+\phi_\xi(x)$ is $a-$harmonic, the high order correctors correct the $a_{\hom}-$harmonic high order polynomials to be $a-$harmonic. It seems the two ways of construction are equivalent although we do not attempt to prove it here. For our purpose, it is more convenient to directly start from the formal expansion. For the recent breakthrough in the direction of regularity theory of random operators in the continuous setting, we refer to \cite{armstrong1,gloria2014regularity,armstrong2,fo,armstrong3,gloria2015optimal,armstrong4}.

%

\medskip

 Our strategy of proving moment estimates in Theorem~\ref{t:mainTh} follows \cite{gloria2011optimal}, i.e., by using a spectral gap inequality (see \eqref{eq:psg} below) and estimating the sensitivity of the correctors to the individual conductance; see also the unpublished work of Naddaf and Spencer \cite{ns}. A key quantity to control is $\partial_e \psi_n^\lambda$, which describes the dependence of the $n-$th order corrector on the conductance $\omega_e$. It involves the first and second order derivatives of the Green's function of $\nabla^*a(x)\nabla$, and we use the $p-$th moment estimates derived in \cite{marahrens2013annealed}, which came from the log-Soblev inequality together with the result of Delmotte and Deuschel on the lower order moment \cite{dd}. One of the difficulties is to obtain some a priori estimate on the gradient $\nabla\psi_n^\lambda$. For the first order corrector, the bound on the second moment of $\nabla \phi_\xi^\lambda$ comes directly from \eqref{eq:1stcor} thanks to the divergence form of its r.h.s. source term. For high order correctors, we prove that the source term can also be expressed in divergence form. 
 Once we have the $p-$th moment estimates on $\nabla\psi_n^\lambda$, Theorem~\ref{t:mainTh} follows from a straightforward induction argument.

\subsection{Two-scale expansions}

As an application of high order correctors, we prove an expansion of solutions to the random PDE \eqref{eq:spde}. If homogenization is viewed as a law of large numbers result, here we are looking for the next order random fluctuations that may or may not lead to a central limit theorem. A classical two-scale expansion indicates that the solutions to \eqref{eq:spde} take the form
\begin{equation}
u_\eps=u_0+\eps u_1+\eps^2 u_2+\ldots,
\label{eq:forex}
\end{equation}
where $u_0,u_1,u_2,\ldots$ are constructed by equating the power-like terms in $\eps$ upon substituting \eqref{eq:forex} into \eqref{eq:spde}. Since \eqref{eq:forex} is only a formal series, it is a priori unclear whether $u_1$ indeed represents the ``correct'' first order fluctuation (it was  shown not true when $d=1$ \cite{Gu2016}), and if it does, we need to understand in which sense \eqref{eq:forex} holds. The second goal of the paper is to give an answer to the above questions and justify the formal two-scale expansion under appropriate dimensional constraint.

To avoid the effects from boundary layers, we work on the equation
\begin{equation}
(\alpha+\nabla_\eps^* a(\frac{x}{\eps})\nabla_\eps) u_\eps(x)=f(x), x\in \eps \Z^d,
\label{eq:heteeq}
\end{equation}
where $\alpha>0$ is a fixed constant and $f(x)\in \C_c^\infty(\R^d)$. The discrete gradient and divergence are defined for $g:\eps\Z^d\to \R$ and $G:\eps\Z^d\to\R^d$ as 
\[
\nabla_\eps g=(\nabla_{\eps,1}g,\ldots,\nabla_{\eps,d}g) \mbox{ with } \nabla_{\eps,i}g(x)=[g(x+\eps e_i)-g(x)]/\eps,
\]
and
\[
\nabla_\eps^* G=\sum_{i=1}^d \nabla_{\eps,i}^* G_i \mbox{ with } \nabla_{\eps,i}^*G_i(x)=[G_i(x-\eps e_i)-G_i(x)]/\eps.
\]

It is well-known that $u_\eps$ converges in a certain sense to $u_{0}$ solving 
\begin{equation}
(\alpha+\nabla_\eps^*a_{\hom}\nabla_\eps) u_0(x)=f(x), x\in \eps \Z^d,
\label{eq:homoeq}
\end{equation}
where the effective coefficient matrix is given by 
\[
a_{\hom}=\la a(\mathrm{I}_d+\nabla\phi)\ra=\bar{a} \mathrm{I}_d \mbox{ with }\nabla\phi=[\nabla\phi_{e_1},\ldots,\nabla\phi_{e_d}],
\]
and $\phi_{e_k}$ is the first order corrector in the direction $e_k$. Our goal is to obtain the first and higher order fluctuations in $u_\eps\to u_0$. By the formal two-scale expansion, the first order correction takes the form 
\[
u_1(x,\frac{x}{\eps})=\sum_{j=1}^d \nabla_{\eps,j}u_0(x)\phi_{e_j}(\frac{x}{\eps}), x\in\eps\Z^d.
\]

\begin{rem}
Typically, one compare $u_\eps$ with $\bar{u}_0$ solving the equation in the continuous space:
\[
(\alpha-\nabla\cdot a_{\hom}\nabla)\bar{u}_0(x)=f(x), x\in\R^d.
\]
In this paper, we do not analyze $u_0-\bar{u}_0$, which is only an error due to discretization.
\end{rem}

For functions $f:\eps\Z^d\times \Omega\to \R^d$, we define the $L^2(\eps\Z^d\times \Omega)$ norm
\begin{equation}
\|f\|_{2,\eps}:=\bigg(\eps^d\sum_{x\in\eps\Z^d} \la |f(x,\omega)|^2\ra\bigg)^{\frac12}.
\end{equation}
The following is our main result for the first order fluctuations.
\begin{thm}
When $d\geq 3$, we have 
\begin{equation}
\|u_\eps-u_0-\eps u_1\|_{2,\eps}=o(\eps)
\label{eq:1stex}
\end{equation}
as $\eps\to 0$.
\label{t:2scale}
\end{thm}

\begin{rem}
Since $\la u_1\ra=0$, Theorem~\ref{t:2scale} in particular implies that the deterministic bias $\la u_\eps\ra-u_0$ vanishes in the order of $\eps$.
\end{rem}

For the errors quantified in the strong sense, by the $\|\cdot\|_{2,\eps}$ norm here, Theorem~\ref{t:2scale} shows that the first order correction is given by the first order corrector that comes from the formal expansion. This is consistent with the pointwise result in the continuous setting \cite{gu-mourrat}, and both results suggest that we do not expect a central limit theorem for $(u_\eps(x)-u_0(x))/\eps$. It should be contrasted with the low dimensional case, e.g., when $d=1$, it was shown in \cite{BP-AA-99,Gu2016} that $(u_\eps(x)-u_0(x))/\sqrt{\eps}$ converges in distribution to a Gaussian process. 
For the errors quantified in the weak sense, i.e., after taking a spatial average with a test function, central limit theorems are obtained for the first order corrector $u_1$ and the solution $u_\eps$ \cite{mourrat2014correlation,MN,gu-mourratmms,gloria2016random,armstrong2016scaling}. Various approximations to the effective coefficient matrix $a_{\hom}$ also exhibit Gaussian fluctuations; see \cite{nolen2011normal,biskup2014central,rossignol2012noise,nolen2014normal,GN2014CLT}.

In higher dimensions, a higher order expansion similar to \eqref{eq:1stex} can also be obtained. 
To make sense of the expansion, the number of terms we are permitted to include in the expansion depends on the dimension in light of Theorem~\ref{t:mainTh}.

\begin{thm}
Fix any $n\geq 2$. When $d\geq 2n+1$, there exists $\{u_k\}_{k=1}^n$ and $\{v_k\}_{k=1}^n$ 
such that $u_k$ is random with $\la u_k\ra=0$, $v_k$ is deterministic, 
and
\[
\bigg\|u_\eps-u_0-\sum_{k=1}^n \eps^ku_k-\sum_{k=1}^n \eps^kv_k\bigg\|_{2,\eps}=o(\eps^n)
\]
as $\eps\to 0$.
\label{t:highex}
\end{thm}

\begin{rem}
Theorem~\ref{t:2scale} can be viewed as a special case of Theorem~\ref{t:highex} when $n=1$, in which case we have $v_1\equiv 0$.
\end{rem}

\begin{rem}
It can be seen from the proof of Theorem~\ref{t:highex} that the $k-$th order random fluctuation $u_k$ takes the form 
\[
u_k(x)= \sum_i g_i(x)\varphi_i(\frac{x}{\eps}),
\]
where $g_i$ is some derivative of $u_0$ and $\varphi_i$ is some zero-mean stationary random field. Therefore, the random fluctuations in homogenization, measured in the strong sense, is a superposition of highly oscillatory random fields.
\end{rem}

Using the high order correctors obtained in Theroem~\ref{t:mainTh}, the proof of Theorem~\ref{t:2scale} and Theorem~\ref{t:highex} mimics the periodic setting. It is well-known that in the \emph{ideal} periodic setting (equation with smooth coefficients posed on the whole space), the formal two-scale expansion as in \eqref{eq:forex} indeed approximates the solution up to arbitrary high order precision, so the message we want to convey here is that one can still hope the expansion to be valid in the random setting, 
provided that we have stationary correctors. We refer to \cite{gerard2011homogenization,gerard2012homogenization} for a careful dealing with boundary layers in the periodic setting.


The previous work on estimating the size of $u_\eps-u_0$ includes e.g. \cite{yurinskii1986averaging,conlon1,conlon2}, with non-optimal exponents or optimal exponents in the small ellipticity regime. With the study of the first order correctors, \cite{gloria2013fluctuation,gloria2014optimal,mourrat2012kantorovich} provides optimal estimates of the error size. The recent preprint \cite{BFFO} uses second order correctors to derive quantitative estimates in the weak spatial norms. The main contribution of Theorem~\ref{t:2scale} and Theorem~\ref{t:highex} is to identify the first and higher order fluctuations in the strong sense, which seems to be the first result of this type. 

\medskip
The paper is organized as follows. In Section~\ref{s:pr}, we present the proof of Theorem~\ref{t:mainTh}. 
 The proofs of Theorem~\ref{t:2scale} and Theorem~\ref{t:highex} are in Section~\ref{s:1stex} and Section~\ref{s:high}. We leave some discussions to Section~\ref{s:compa}.

\subsection{Notations}

\begin{itemize}
\item
For any $e=(\underline{e},\bar{e})\in \B$ and $f:\Z^d\to \R$, we will write 
\[
\nabla f(e)=f(\bar{e})-f(\underline{e}),
\]
and for any $\xi\in\R^d$, 
\[
\xi(e)=\xi\cdot e.
\]
\item
For $x\in \Z^d$, we define
\[  
|x|_*=|x|+2.
\]
\item 
We write $a\les b$ when $a\leq Cb$ with the constant $C$ independent of the spatial variables $y,z$ and the edge variables $b,e$.
\item
For $x,y,c>0$, we write 
\[
x\les \frac{1}{y^{c-}}
\]
 if for any $\delta>0$, there exists $C_\delta>0$ such that 
 \[
 x\leq C_\delta \frac{1}{y^{c-\delta}}.
 \]
 \item
For $f:\Z^d\times\Z^d\to \R$, $x\in \Z^d$ and $e,b\in\B$,  we write
\[
\nabla f(x,e)=f(x,\bar{e})-f(x,\un{e}),
\]
and
\[
 \nabla\nabla f(b,e)=\nabla f(\bar{b},e)-\nabla f(\un{b},e).
\]
\item
We use $\|\cdot\|_p$ to denote the $L^p(\Omega)$ norm. Recall that $\|\cdot\|_{2,\eps}$ denote the $L^2(\eps\Z^d\times \Omega)$ norm.
\item We use $G^\lambda,\G^\lambda$ to denote the Green's function of $\lambda+\nabla^*a\nabla, \lambda+\nabla^*\nabla$ on $\Z^d$ and $\cG^\lambda$ to denote the Green's function of $\lambda-\Delta$ on $\R^d$.
\end{itemize}

\section{Moment bounds of high order correctors}
\label{s:pr}

We first review the classical two-scale expansion of equations in divergence form, then present a proof of Theorem~\ref{t:mainTh}.

\subsection{Two-scale expansions} 
\label{s:2scale}

Consider the equation on $\R^d$:
\[
-\nabla\cdot a(\frac{x}{\eps})\nabla u_\eps(x)=f(x),
\]
where $a:\R^d\to \R^{d\times d}$ is a stationary random field. We introduce the fast variable $y=\frac{x}{\eps}$ and write $\nabla=\nabla_x+\frac{1}{\eps}\nabla_y$. By further assuming 
\begin{equation}
u_\eps(x)=u_0(x)+\eps u_1(x,y)+\eps^2 u_2(x,y)+\ldots,
\label{eq:twoscale}
\end{equation}
we obtain
\[
-(\nabla_x+\frac{1}{\eps}\nabla_y)\cdot a(y)(\nabla_x+\frac{1}{\eps}\nabla_y) (u_0+\eps u_1+\ldots+\eps^n u_n+\ldots)=f.
\]
By matching the order of $\eps$, we get equations satisfied by $u_n$.

When $n=1$,
\begin{equation}
\nabla_y\cdot a(y)\nabla_y u_1+\nabla_y\cdot a(y)\nabla_x u_0=0.
\label{eq:1stcorEq}
\end{equation}

When $n\geq 2$,
\begin{equation}
\begin{aligned}
&\nabla_y \cdot a(y)\nabla_y u_n+\nabla_y\cdot a(y)\nabla_x u_{n-1}+\nabla_x \cdot a(y)\nabla_y u_{n-1}+\nabla_x \cdot a(y)\nabla_x u_{n-2}\\
=&\la \nabla_x \cdot a(y)\nabla_y u_{n-1}+\nabla_x \cdot a(y)\nabla_x u_{n-2}\ra.
\label{eq:highcorEq}
\end{aligned}
\end{equation}
Note that in general, the corrector equations take the form 
\[
\nabla_y\cdot a(y)\nabla_y u_n(x,y)=F
\]
and a solvability condition requires that the source satisfies $\la F\ra=0$, so we added some expectations to the r.h.s. of \eqref{eq:highcorEq}. 

For \eqref{eq:1stcorEq}, we can write
\begin{equation}
u_1(x,y)=\sum_{k=1}^d \partial_k u_0(x)\phi_{e_k}(y)=\nabla u_0(x)\cdot \phi(y),
\label{eq:u1}
\end{equation}
with $\phi=[\phi_{e_1},\ldots,\phi_{e_d}]$ and $\phi_{e_k}$ the first order corrector solving
\[
\nabla_y\cdot a(y)(\nabla\phi_{e_k}+e_k)=0.
\]

The equation satisfied by the second order corrector takes the form 
\begin{equation}
\begin{aligned}
\nabla_y \cdot a(y)\nabla_y u_2
=&\nabla_x \cdot a_{\hom}\nabla_x u_0(x)-\nabla_x\cdot  a(y)\nabla\phi(y)\nabla u_0(x)\\
&-\nabla_x\cdot a(y)\nabla_x u_0(x)-\nabla_y\cdot a(y)\nabla^2 u_0(x)\phi(y),
\end{aligned}
\label{eq:u20}
\end{equation}
where $\nabla^2u_0$ is the Hessian matrix of $u_0$ and the matrix $\nabla\phi=[\nabla\phi_{e_1},\ldots,\nabla\phi_{e_d}]$. The r.h.s. of the above expression has mean zero since the homogenization matrix $a_{\hom}$ is given by 
\[
a_{\hom}=\la a(y)(\mathrm{I}_d+\nabla\phi(y))\ra.
\]

By induction, it is clear that $u_n$ is a linear combination of $\partial_\alpha u_0$, where $\alpha$ is a multi-index with $|\alpha|=n$, so we write 
\[
u_n(x,y)=\sum_\alpha \phi_\alpha(y)\partial_\alpha u_0(x),
\]
and those $\phi_\alpha$ are the $n-$th order correctors.

By \eqref{eq:highcorEq}, we further observe that $\phi_\alpha$ with $|\alpha|=n$ is a linear combination of $\psi_n$, which solves the following three types of equations:
\begin{equation}
\nabla\cdot a(y)\nabla \psi_n(y)=\left\{\begin{array}{l}
\partial_i (a_{ij}(y)\psi_{n-1}(y)),\\
a_{ij}(y) \partial_j \psi_{n-1}(y)-\la a_{ij}(y)\partial_j \psi_{n-1}(y)\ra,\\
a_{ij}(y)\psi_{n-2}(y)-\la a_{ij}(y)\psi_{n-2}(y)\ra.
\end{array}
\right.
\label{eq:corn}
\end{equation}
Here $i,j=1,\ldots,d$, and we have the following correspondence between the r.h.s. of \eqref{eq:corn} and \eqref{eq:highcorEq} :
\[
\begin{aligned}
\partial_i (a_{ij}(y)\psi_{n-1}(y)) &\longleftrightarrow \nabla_y\cdot a(y)\nabla_x u_{n-1}(x,y), \\
a_{ij}(y) \partial_j \psi_{n-1}(y)-\la a_{ij}(y)\partial_j \psi_{n-1}(y)\ra &\longleftrightarrow \nabla_x\cdot a(y)\nabla_y u_{n-1}(x,y)-\la\nabla_x\cdot a(y)\nabla_y u_{n-1}(x,y)\ra,\\
a_{ij}(y)\psi_{n-2}(y)-\la a_{ij}(y)\psi_{n-2}(y)\ra &\longleftrightarrow \nabla_x \cdot a(y)\nabla_x u_{n-2}(x,y)-\la\nabla_x \cdot a(y)\nabla_x u_{n-2}(x,y)\ra.
\end{aligned}
\]
In other words, since \eqref{eq:highcorEq} is a linear equation, we have decomposed it into finitely many ``small equations'' written in the generic form of \eqref{eq:corn}.

\subsection{A proof of Theorem~\ref{t:mainTh} by induction}
\label{s:proof}

Recalling that the coefficient $a(x)=\mathrm{diag}(a_1(x),\ldots,a_d(x))$ is a diagonal matrix for $x\in \Z^d$, so given \eqref{eq:corn} in the continuous setting, we consider corrector equations on $\Z^d$ of the following form:
\begin{equation}
\nabla^*a(y)\nabla \psi_n(y)=\left\{\begin{array}{l}
\nabla_i^* (a_i(y)\psi_{n-1}(y)),\\
a_i(y) \nabla_i\psi_{n-1}(y)-\la a_i(y)\nabla_i \psi_{n-1}(y)\ra,\\
a_{i}(y)\psi_{n-2}(y)-\la a_{i}(y)\psi_{n-2}(y)\ra.
\end{array}
\right.
\label{eq:corndis}
\end{equation}

For any $\lambda>0$, we regularize \eqref{eq:corndis} by adding a massive term and define $\psi_n^\lambda$ as the unique solution to 
\begin{equation}
(\lambda+\nabla^*a(y)\nabla) \psi_n^\lambda(y)=\left\{\begin{array}{l}
\nabla_i^* (a_i(y)\psi_{n-1}(y)),\\
a_i(y) \nabla_i\psi_{n-1}(y)-\la a_i(y)\nabla_i \psi_{n-1}(y)\ra,\\
a_{i}(y)\psi_{n-2}(y)-\la a_{i}(y)\psi_{n-2}(y)\ra.
\end{array}
\right.
\label{eq:corndislambda}
\end{equation}
To prove Theorem~\ref{t:mainTh}, we only analyze \eqref{eq:corndislambda} when $d\geq 2n-1$. Since our proof is based on induction, we do not worry about the meaning of  the r.h.s. of \eqref{eq:corndislambda} for the moment.

Let $\psi_1$ denote the stationary first order correctors $\phi_{e_i}$ when $d\geq 3$. We make the statement $\mathscr{I}_n$ for $n\geq 1$:

\medskip

\textbf{Statement $\mathscr{I}_1$}: when $d\geq 3$, for any $p\geq 1$, we have
\begin{equation}
\|\partial_e \psi_1(y)\|_p\les \frac{1}{|y-\un{e}|_*^{d-1}},
\label{eq:eq01}
\end{equation}
 and
 \begin{equation}
\|\partial_e \nabla\psi_1(b)\|_p\les \frac{1}{|\un{b}-\un{e}|_*^{d}},
\label{eq:eq02}
\end{equation}
with the proportional constant independent of $y\in \Z^d$ and $b,e\in\B$.

\textbf{Statement $\mathscr{I}_n, n\geq 2$}: when $d\geq 2n-1$, for any $p\geq 1$, we have
\begin{equation}
\|\partial_e \psi_n^\lambda(y)\|_p\les \frac{1}{|y-\un{e}|_*^{(d-n)-}},
\label{eq:eq1}
\end{equation}
 and
 \begin{equation}
\|\partial_e \nabla\psi_n^\lambda(b)\|_p\les \frac{1}{|\un{b}-\un{e}|_*^{(d-n+1)-}},
\label{eq:eq2}
\end{equation}
with the proportional constant independent of $\lambda>0$, $y\in \Z^d$ and $b,e\in\B$.

\medskip

The following result is the main result in this section.

\begin{prop}
For any $n\geq 2$, $\mathscr{I}_{n-2}+\mathscr{I}_{n-1}$ implies $\mathscr{I}_{n}$, where we take $\mathscr{I}_0$ as an empty statement.
\label{p:induction}
\end{prop}

First, we note that $\mathscr{I}_1$ holds. For the first order correctors $\psi_1\in\{\phi_{e_k}\}_{k=1,\ldots,d}$, 
\[
\partial_e \phi_{e_k}^\lambda(y)=-\nabla G^\lambda(y,e)(\nabla\phi_{e_k}+e_k)(e),
\]
and
\[
\partial_e \nabla \phi_{e_k}^\lambda(b)=-\nabla\nabla G^\lambda(b,e)(\nabla\phi_{e_k}+e_k)(e),
\]
where $G^\lambda$ is the Green's function of $\lambda+\nabla^*a\nabla$. By \cite[Theorem 1]{marahrens2013annealed}, we have
\begin{equation}
\|\nabla G^\lambda(y,e)\|_p\les \frac{1}{|y-\un{e}|_*^{d-1}} \mbox{ and } \|\nabla\nabla  G^\lambda(b,e)\|_p\les \frac{1}{|\un{b}-\un{e}|_*^{d}},
\label{eq:bdgr}
\end{equation}
which implies 
\begin{equation}
\|\partial_e \phi_{e_k}(y)\|_p\les \frac{1}{|y-\un{e}|_*^{d-1}} \mbox{ and } \|\partial_e \nabla\phi_{e_k}(b)\|_p\les \frac{1}{|\un{b}-\un{e}|_*^d}.
\label{eq:j1}
\end{equation}

By Proposition~\ref{p:induction}, $\mathscr{I}_n$ holds for $n\geq 2$, and this implies Theorem~\ref{t:mainTh}. First, the log-Soblev inequality \eqref{eq:logsob} implies the Spectral-Gap inequality:
\begin{equation}
\la |\zeta|^2\ra\les \lla \sum_e |\partial_e \zeta|^2\rra
\label{eq:sgin}
\end{equation}
for zero-mean $\zeta$, which can be generalized to control the $2p-$th moment for any $p>1$ \cite[Lemma 2]{gloria2013quantification}:
\begin{equation}
\la |\zeta|^{2p}\ra \les \lla \left(\sum_e |\partial_e\zeta|^2\right)^p\rra.
\label{eq:psg}
\end{equation}

\begin{rem}
The result presented in \cite[Lemma 2]{gloria2013quantification} is for a different type of vertical derivative, but the same proof applies in our case.
\end{rem}

Now we apply \eqref{eq:psg} to $\psi_n^\lambda$, and use $\mathscr{I}_n$ to derive
\[
\begin{aligned}
\|\psi_n^\lambda(0)\|_{2p}\les \lla\left(\sum_e |\partial_e\psi_n^\lambda(0)|^2\right)^p\rra^{\frac{1}{2p}} \les &\left(\sum_e \|\partial_e \psi_n^\lambda(0)\|_{2p}^2\right)^{\frac12}\\
\les & \left(\sum_e  \frac{1}{|\un{e}|_*^{(2d-2n)-}}\right)^{\frac12}\les 1
\end{aligned}
\]
when $d\geq 2n+1$.

Similarly, we apply \eqref{eq:psg} to $\nabla\psi_n^\lambda$, and use $\mathscr{I}_{n}$ to derive
\[
\begin{aligned}
\|\nabla\psi_n^\lambda(b)\|_{2p}\les \lla\left(\sum_e |\partial_e\nabla \psi_n^\lambda(b)|^2\right)^p\rra^{\frac{1}{2p}} \les &\left(\sum_e \|\partial_e \nabla\psi_n^\lambda(b)\|_{2p}^2\right)^{\frac12}\\
\les & \left(\sum_e  \frac{1}{|\un{b}-\un{e}|_*^{(2d-2n+2)-}}\right)^{\frac12}\les 1
\end{aligned}
\]
when $d\geq 2n-1$.

By the uniform bound on $\|\psi_n^\lambda\|_{2p}$ and $\|\nabla\psi_n^\lambda(b)\|_{2p}$, it is standard to extract a subsequence as $\lambda\to 0$ to get the existence of a stationary corrector $\psi_n$ when $d\geq 2n+1$ and the existence of a stationary gradient $\nabla\psi_n$ when $d\geq 2n-1$. For the convenience of reader, we present the details in Appendix~\ref{s:exst}.

\begin{rem}
It is clear that if $\psi_n^\lambda(0)$ converges to $\psi_n(0)$ weakly in $L^p(\Omega)$, then $\partial_e\psi_n^\lambda(0)$ converges to $\partial_e\psi_n(0)$ weakly in $L^p(\Omega)$, thus the same estimates in $\mathscr{I}_n$ holds for $\psi_n$ when $d\geq 2n+1$. The same discussion applies to $\nabla\psi_n$.
\label{r:bscor}
\end{rem}

The proof of Proposition~\ref{p:induction} is a straightforward calculation once we assume the high moment bounds for $\nabla\psi_{n}^\lambda$: \begin{equation}
\mathscr{I}_{n-2}+\mathscr{I}_{n-1}\Rightarrow \mbox{ when } d\geq 2n-1,
\mbox{ for any } p\geq 1, 
 \|\nabla \psi_{n}^\lambda(b)\|_p\les 1.
\label{eq:mmgra}
\end{equation}
We will leave the proof of 
$\eqref{eq:mmgra}$ to the next section.



\medskip

\emph{Proof of Proposition~\ref{p:induction}.} By the previous discussion, when $d\geq 2n-1$, we have stationary $(n-2)-$th and $(n-1)-$th order corrector, so the equation of $\psi_n^\lambda$ given by \eqref{eq:corndislambda} is well-defined:
\[
(\lambda+\nabla^*a(y)\nabla) \psi_n^\lambda(y)=\left\{\begin{array}{l}
\nabla_i^* (a_i(y)\psi_{n-1}(y)),\\
a_i(y) \nabla_i\psi_{n-1}(y)-\la a_i(y)\nabla_i \psi_{n-1}(y)\ra,\\
a_{i}(y)\psi_{n-2}(y)-\la a_{i}(y)\psi_{n-2}(y)\ra.
\end{array}
\right.
\]
We take $\partial_e$ on both sides to obtain
\[
(\lambda+\nabla^* a(y)\nabla) \partial_e \psi_n^\lambda(y)=(1_{y=\un{e}}-1_{y=\bar{e}})\nabla\psi_n^\lambda(e)+\left\{\begin{array}{l}
\nabla_i^* [\partial_e(a_i(y)\psi_{n-1}(y))],\\
\partial_e(a_i(y) \nabla_i\psi_{n-1}(y)),\\
\partial_e(a_{i}(y)\psi_{n-2}(y)).
\end{array}
\right.
\]
By the Green's function representation, we write
\begin{equation}
\partial_e \psi_n^\lambda(y)=-\nabla G^\lambda(y,e)\nabla\psi_n^\lambda(e)+\left\{\begin{array}{l}
\sum_{z}\nabla_i G^\lambda(y,z) [\partial_e(a_i(z)\psi_{n-1}(z))],\\
\sum_z G^\lambda(y,z)\partial_e(a_i(z) \nabla_i\psi_{n-1}(z)),\\
\sum_z G^\lambda(y,z)\partial_e(a_{i}(z)\psi_{n-2}(z)),
\end{array}
\right.
\label{eq:depsin}
\end{equation}
which implies that for any $p\geq 1$:
\[
\|\partial_e \psi_n^\lambda(y)\|_p\leq \|\nabla G^\lambda(y,e)\nabla\psi_n^\lambda(e)\|_p+\left\{\begin{array}{l}
\sum_{z}\|\nabla_i G^\lambda(y,z) [\partial_e(a_i(z)\psi_{n-1}(z))]\|_p,\\
\sum_z \|G^\lambda(y,z)\partial_e(a_i(z) \nabla_i\psi_{n-1}(z))\|_p,\\
\sum_z \|G^\lambda(y,z)\partial_e(a_{i}(z)\psi_{n-2}(z))\|_p.
\end{array}
\right.
\]
%

Since $\partial_e a_i(z)=0$ for $e\neq (z,z+e_i)$, it is clear by $\mathscr{I}_{n-2},\mathscr{I}_{n-1}$ and Remark~\ref{r:bscor} that 
\[
\begin{aligned}
\|\partial_e(a_i(z)\psi_{n-1}(z))\|_p &\les \frac{1}{|z-\un{e}|_*^{(d-n+1)-}},\\
\|\partial_e(a_i(z)\nabla_j\psi_{n-1}(z))\|_p &\les \frac{1}{|z-\un{e}|_*^{(d-n+2)-}},\\
\|\partial_e(a_i(z)\psi_{n-2}(z))\|_p &\les \frac{1}{|z-\un{e}|_*^{(d-n+2)-}}.
\end{aligned}
\]
 Note that when $n=2$, $\psi_{n-2}(z)$ is a constant as can be seen from \eqref{eq:u20}, so the above estimates still holds. An application of H\"older inequality together with \eqref{eq:bdgr} and \eqref{eq:mmgra} leads to 
\[
\|\partial_e \psi_n^\lambda(y)\|_p\les \frac{1}{|y-\un{e}|_*^{d-1}}+\left\{\begin{array}{l}
\sum_{z}\frac{1}{|y-z|_*^{d-1}}\frac{1}{|z-\un{e}|_*^{(d-n+1)-}},\\
\sum_z \frac{1}{|y-z|_*^{d-2}}\frac{1}{|z-\un{e}|_*^{(d-n+2)-}},\\
\sum_z\frac{1}{|y-z|_*^{d-2}}\frac{1}{|z-\un{e}|_*^{(d-n+2)-}}.
\end{array}
\right.
\]
By the discrete convolution inequality \cite[Lemma A.6]{gu-mourratmms}, we obtain
\begin{equation}
\|\partial_e \psi_n^\lambda(y)\|_p\les \frac{1}{|y-\un{e}|_*^{(d-n)-}}.
\label{eq:psin1}
\end{equation}

By \eqref{eq:depsin}, we also have
\begin{equation}
\partial_e \nabla \psi_n^\lambda (b)=-\nabla \nabla G^\lambda (b,e)\nabla\psi_n^\lambda(e)+\left\{\begin{array}{l}
\sum_{z}\nabla \nabla_i G^\lambda(b,z) [\partial_e(a_i(z)\psi_{n-1}(z))],\\
\sum_z \nabla G^\lambda(b,z)\partial_e(a_i(z) \nabla_i\psi_{n-1}(z)),\\
\sum_z \nabla G^\lambda(b,z)\partial_e(a_{i}(z)\psi_{n-2}(z)),
\end{array}
\right.
\label{eq:degrpsin}
\end{equation}
so the same discussion as above gives 
\begin{equation}
\|\partial_e \nabla \psi_n^\lambda(b)\|_p\les \frac{1}{|\un{b}-\un{e}|_*^{(d-n+1)-}}.
\label{eq:psin2}
\end{equation}
The statement $\mathscr{I}_n$ consists of \eqref{eq:psin1} and \eqref{eq:psin2}, so the proof of Proposition~\ref{p:induction} is complete.

\begin{rem}
The reason we choose to have $(d-n)-$ and $(d-n+1)-$ as the exponents in \eqref{eq:eq1} and \eqref{eq:eq2} is due to the following discrete convolution inequality: for any $\alpha\in (0,d)$,
\begin{equation}
\sum_z \frac{1}{|z|_*^\alpha}\frac{1}{|y-z|_*^d} \les \frac{\log |y|_*}{|y|_*^\alpha}\les \frac{1}{|y|_*^{\alpha-}}.
\label{eq:dcon}
\end{equation}
As can be seen from \eqref{eq:degrpsin}, the term $\nabla \nabla_i G^\lambda (b,z)$ produces a factor of $|\un{b}-z|_*^{-d}$, and we use ``$-$'' to absorb the logarithmic factor when applying \eqref{eq:dcon}.
\end{rem}

\subsection{Estimates on gradient of correctors: proof of \eqref{eq:mmgra}} 
\label{s:boundg}


To prove 
\[
\|\nabla\psi_n^\lambda(b)\|_p\les 1
\]
 for any $p\geq 1$, we first consider $p=2$, then derive it for any $p>2$.

\subsubsection{$p=2$.} 
\label{s:q2}
We write the equation \eqref{eq:corndislambda} as
\begin{equation}
(\lambda+\nabla^* a(y)\nabla)\psi_n^\lambda(y)=F(y)=\left\{\begin{array}{l}
\nabla_i^* (a_i(y)\psi_{n-1}(y)),\\
a_i(y) \nabla_i\psi_{n-1}(y)-\la a_i(y)\nabla_i \psi_{n-1}(y)\ra,\\
a_{i}(y)\psi_{n-2}(y)-\la a_{i}(y)\psi_{n-2}(y)\ra.
\end{array}
\right.
\label{eq:corgene}
\end{equation}
By $\mathscr{I}_{n-2}+\mathscr{I}_{n-1}$, when $d\geq 2n-1$, $F$ is a zero-mean stationary random field, and
\[
\|\psi_{n-1}\|_q+\|\nabla \psi_{n-1}\|_q+\|\psi_{n-2}\|_q\les1
\]
for any $q\geq 1$.
We claim that there exists a stationary random field $\Psi$ such that $\Psi(0)\in L^q(\Omega)$ for any $q\geq 1$ and 
\begin{equation}
\nabla^* \Psi=F.
\label{eq:dipsif}
\end{equation}
Using the fact $a_i> \delta>0$, we conclude from \eqref{eq:dipsif} that
\[
\|\nabla\psi_n^\lambda\|_2\les \|\Psi\|_2\les 1.
\]

For $F=\nabla_i^* (a_i(y)\psi_{n-1}(y))$, we only need to choose 
\[
\Psi=\Psi_ie_i \mbox{ with }\Psi_i=a_i\psi_{n-1}.
\]

For the other two cases, we consider the equation with $k=1,\ldots,d$
\begin{equation}
(\lambda+\nabla^*\nabla)\Psi_k^\lambda=\nabla_k F,
\label{eq:plala}
\end{equation}
and by $\mathscr{I}_{n-2}+\mathscr{I}_{n-1}$, we have
\[
\begin{aligned}
\|\partial_e \Psi_k^\lambda(y)\|_q\leq \sum_z \|\nabla_k^* \ov{G^\lambda}(y,z)\partial_eF(z)\|_q\les& \sum_z \frac{1}{|y-z|_*^{d-1}}\|\partial_e F(z)\|_q\\
\les &\sum_z \frac{1}{|y-z|_*^{d-1}}\frac{1}{|z-\un{e}|_*^{(d-n+2)-}}\\
\les &\frac{1}{|y-\un{e}|_*^{(d-n+1)-}},
\end{aligned}
\]
where we used $\ov{G^\lambda}$ to denote the Green's function of $\lambda+\nabla^*\nabla$. Thus, by the same application of the spectral gap estimate \eqref{eq:psg}, we have 
\[
\|\Psi_k^\lambda\|_q\les 1
\]
when $d\geq 2n-1$. By Lemma~\ref{l:stacor}, there exists a stationary random field $\Psi_k(y)$ such that $\Psi_k(0)\in L^q(\Omega)$ and $\Psi_k^\lambda(0) \to \Psi_k(0)$ weakly in $L^q(\Omega)$. Let $\Psi=(\Psi_1,\ldots,\Psi_d)$, and we claim 
\[
\nabla^*\Psi=F.
\]
From \eqref{eq:plala}, we have 
\[
(\lambda+\nabla^*\nabla)\nabla_k^*\Psi_k^\lambda=\nabla_k^*\nabla_kF,
\]
so 
\[
(\lambda+\nabla^*\nabla)\nabla^*\Psi^\lambda=\nabla^*\nabla F
\]
by summing over $k$. Let $\lambda\to 0$, it is clear that 
\[
\nabla^*\nabla (\nabla^*\Psi-F)=0.
\]
By ergodicity, we have $\nabla^*\Psi-F$ is a constant, but since it has mean zero, we conclude $\nabla^*\Psi=F$. The proof is complete for $p=2$.

\subsubsection{$p>2$.} 
\label{s:qg2}
By \cite[Lemma 4]{marahrens2013annealed}, the log-Soblev inequality implies a type of inverse H\"older inequality: fix any $q>1$, for any $\delta>0$, there exists $C_\delta>0$ such that
\begin{equation}
\la |\xi|^{2q}\rangle^{\frac{1}{2q}}\leq C_\delta\langle |\xi|\rangle+\delta\lla \left(\sum_e |\partial_e \xi|^2\right)^{q}\rra^{\frac{1}{2q}}.
\label{eq:inho}
\end{equation}

\begin{rem}
The result in \cite[Lemma 4]{marahrens2013annealed} is for a different type of derivative, but the same proof applies in our case.
\end{rem}

In view of \eqref{eq:inho} and the bound on $\|\nabla\psi_n^\lambda\|_2$, in order to bound $\|\nabla\psi_n^\lambda\|_q$ for any $q>1$, it suffices to prove 
\begin{equation}
\lla \left(\sum_e |\partial_e \nabla\psi_n^\lambda|^2\right)^q\rra^{\frac{1}{2q}}\les \la |\nabla\psi_n^\lambda|^{2q}\rangle^{\frac{1}{2q}}+1,
\label{eq:deto}
\end{equation}
and choose $\delta$ sufficiently small. By recalling \eqref{eq:degrpsin}, we have
\begin{equation}
|\partial_e \nabla\psi_n^\lambda(b)|^2\les |\nabla\nabla G^\lambda(b,e)\nabla\psi_n^\lambda(e)|^2+I_e^2,
\label{eq:decom}
\end{equation}
where 
\[
I_e:=\left\{\begin{array}{l}
\sum_{z}\nabla \nabla_i G^\lambda(b,z) [\partial_e(a_i(z)\psi_{n-1}(z))],\\
\sum_z \nabla G^\lambda(b,z)\partial_e(a_i(z) \nabla_i\psi_{n-1}(z)),\\
\sum_z \nabla G^\lambda(b,z)\partial_e(a_{i}(z)\psi_{n-2}(z)).
\end{array}
\right.
\]
We first have
\[
\lla \left(\sum_e I_e^2\right)^q\rra\les \left(\sum_e \|I_e\|_{2q}^2\right)^q\les 1
\]
by $\mathscr{I}_{n-2}+\mathscr{I}_{n-1}$ and the fact that $d\geq 2n-1$. 
For the first term on the r.h.s. of \eqref{eq:decom}, we have 
\[
\begin{aligned}
 \left(\sum_e|\nabla\nabla G^\lambda(b,e)\nabla\psi_n^\lambda(e)|^2\right)^q  \les &\left(\sum_e |\nabla\nabla G^\lambda(b,e)|^2\right)^{q-1} \sum_e |\nabla\nabla G^\lambda(b,e)|^2|\nabla\psi_n^\lambda(e)|^{2q}\\
 \les &\sum_e |\nabla\nabla G^\lambda(b,e)|^2|\nabla\psi_n^\lambda(e)|^{2q},
 \end{aligned}
\]
where we used the quenched bound \cite[Equation (39)]{marahrens2013annealed}
\begin{equation}
\sum_e |\nabla\nabla G^\lambda(b,e)|^2\les 1.
\label{eq:qubd}
\end{equation}  
Thus we have
\[
\lla \left(\sum_e|\nabla\nabla G^\lambda(b,e)\nabla\psi_n^\lambda(e)|^2\right)^q \rra \les\sum_e\la |\nabla\nabla G^\lambda(b,e)|^2|\nabla\psi_n^\lambda(e)|^{2q}\ra.
\] 
By stationarity we can write 
\[
\la |\nabla\nabla G^\lambda(b,e)|^2|\nabla\psi_n^\lambda(e)|^{2q}\ra=\la |\nabla\nabla G^\lambda(b-\un{e},e-\un{e})|^2|\nabla\psi_n^\lambda(e-\un{e})|^{2q}\ra,
\]
so by considering the summation over edges of the same direction, $\nabla\psi_n^\lambda(e-\un{e})$ is unchanged over $e$, and using \eqref{eq:qubd} again we have 
\[
\sum_e\la |\nabla\nabla G^\lambda(b,e)|^2|\nabla\psi_n^\lambda(e)|^{2q}\ra
\les \la |\nabla\psi_n^\lambda|^{2q}\ra.
\]
To summarize, we have proved that
\[
\lla \left(\sum_e |\partial_e \nabla\psi_n^\lambda|^2\right)^q\rra\les  \langle |\nabla\psi_n^\lambda|^{2q}\ra+1,
\]
which leads to \eqref{eq:deto} and completes the proof.

\section{First order fluctuations when $d\geq 3$}
\label{s:1stex}

The proof of Theorem~\ref{t:2scale} follows the ideas of \cite[Section 6]{papanicolaou1979boundary}, which itself is analogous to the periodic case. Since the first order correctors are used to prove the convergence of $u_\eps\to u_0$, it is natural to consider using the second order correctors to obtain the first order fluctuations. 

We define the remainder in the expansion as
\begin{equation}
z_\eps(x)=u_\eps(x)-u_0(x)-\eps u_1(x,y)-\eps^2 u_2(x,y),
\label{eq:remainder}
\end{equation}
where $u_1,u_2$ will be constructed by the first and second order correctors later. Since $d\geq 3$, we have stationary first order correctors, so $\eps u_1\sim O(\eps)$. We do not necessarily have stationary second order correctors (by Theorem~\ref{t:mainTh}, we need at least $d=5$), but they have zero-mean stationary gradients, so we can expect the second order correctors to grow sublinearly, which implies $|\eps^2u_2(x,y)|\ll \eps$. The rest is to analyze $z_\eps$.

We first construct $u_1$ and $u_2$ for our purpose. Then we prove $\eps^2 u_2$ and $z_\eps$ are both of order $o(\eps)$. 

\subsection{Construction of $u_1,u_2$}
\label{s:conu1u2}

Similar to the continuous setting, we introduce the fast variable $y=x/\eps\in \Z^d$ when $x\in \eps\Z^d$. In the discrete setting, the Leibniz rule is different, and for any function $f:\eps\Z^d\to \R$ and $g:\Z^d\to \R$, we have
\[
\begin{aligned}
\nabla_{\eps,i}(f(x)g(y))=&[f(x+\eps e_i)g(y+e_i)-f(x)g(y)]/\eps\\
=&\left\{\begin{array}{l}
\nabla_{\eps,i}f(x) g(y)+f(x+\eps e_i)\frac{1}{\eps}\nabla_ig(y),\\
\nabla_{\eps,i}f(x)g(y+e_i)+f(x)\frac{1}{\eps}\nabla_ig(y),
\end{array}
\right.
\end{aligned}
\]
and
\[
\begin{aligned}
\nabla_{\eps,i}^*(f(x)g(y))=&[f(x-\eps e_i)g(y-e_i)-f(x)g(y)]/\eps\\
=&\left\{\begin{array}{l}
\nabla_{\eps,i}^*f(x) g(y)+f(x-\eps e_i)\frac{1}{\eps}\nabla_i^*g(y),\\
\nabla_{\eps,i}^*f(x)g(y-e_i)+f(x)\frac{1}{\eps}\nabla_i^*g(y).
\end{array}
\right.
\end{aligned}
\]
There are two different expressions for the Leibniz rule, and for our purpose, we choose the one that does not change the microscopic variable $y$:
\begin{equation}
\begin{aligned}
\nabla_{\eps,i}(f(x)g(y))&=&\nabla_{\eps,i}f(x) g(y)+f(x+\eps e_i)\frac{1}{\eps}\nabla_ig(y),\\
\nabla_{\eps,i}^*(f(x)g(y))&=&\nabla_{\eps,i}^*f(x) g(y)+f(x-\eps e_i)\frac{1}{\eps}\nabla_i^*g(y).
\label{eq:disle}
\end{aligned}
\end{equation}

We construct $u_1,u_2$ by expanding $\nabla_\eps^*a(y)\nabla_\eps(u_0+\eps u_1+\eps^2u_2)$ with \eqref{eq:disle} and equating the like power of $\eps$, which is similar to what we presented in Section~\ref{s:2scale}. Due to the lattice effect, the expressions of $u_1,u_2$ will be different from the continuous setting, so we present the details of the calculations here.

We first have
\begin{equation}
\nabla_\eps^*a(y) \nabla_\eps u_0(x)=\sum_{i=1}^d a_i(y)\nabla_{\eps,i}^*\nabla_{\eps,i} u_0(x)+\sum_{i=1}^d\frac{1}{\eps}\nabla_i^*a_i(y)\nabla_{\eps,i}u_0(x-\eps e_i),
\label{eq:nau0}
\end{equation}
and we define 
\begin{equation}
u_1(x,y)=\sum_{j=1}^d \nabla_{\eps,j}u_0(x-\eps e_j) \phi_{e_j}(y),
\label{eq:u1}
\end{equation}
where we recall that $\phi_{e_j}$ is the first order corrector in the direction of $e_j$, which is a mean-zero stationary random field when $d\geq 3$ and satisfies 
\[
\nabla^*a(y)(\nabla\phi_{e_j}(y)+e_j)=0.
\]

Next, we have
\[
\begin{aligned}
\nabla_\eps^*a(y)\nabla_\eps(\eps u_1(x,y))=&\sum_{i,j=1}^d \eps \nabla_{\eps,i}^* a_i(y)\nabla_{\eps,i}(\nabla_{\eps,j}u_0(x-\eps e_j)\phi_{e_j}(y))\\
=&\sum_{i,j=1}^d \eps\nabla_{\eps,i}^*a_i(y)(\nabla_{\eps,i}\nabla_{\eps,j}u_0(x-\eps e_j) \phi_{e_j}(y))\\
+&\sum_{i,j=1}^d \nabla_{\eps,i}^* a_i(y)(\nabla_{\eps,j}u_0(x-\eps e_j+\eps e_i)\nabla_{i}\phi_{e_j}(y)):=I_1+I_2.
\end{aligned}
\]
For $I_1$, we have
\[
\begin{aligned}
I_1=&\sum_{i,j=1}^d \eps \nabla_{\eps,i}^*\nabla_{\eps,i}\nabla_{\eps,j}u_0(x-\eps e_j)a_i(y)\phi_{e_j}(y)\\
+&\sum_{i,j=1}^d \nabla_{\eps,i}\nabla_{\eps,j}u_0(x-\eps e_j-\eps e_i)\nabla_i^*(a_i(y)\phi_{e_j}(y)).
\end{aligned}
\]
For $I_2$, we have
\[
\begin{aligned}
I_2=&\sum_{i,j=1}^d \nabla_{\eps,i}^* \nabla_{\eps,j}u_0(x-\eps e_j+\eps e_i)a_i(y)\nabla_i\phi_{e_j}(y)\\
+&\sum_{i,j=1}^d \frac{1}{\eps}\nabla_{\eps,j}u_0(x-\eps e_j) \nabla_i^*(a_i(y)\nabla_i\phi_{e_j}(y)).
\end{aligned}
\]
By the equation satisfied by $\phi_{e_j}$, we have the second term on the r.h.s. of the above display equals to 
\[
\sum_{i,j=1}^d \frac{1}{\eps}\nabla_{\eps,j}u_0(x-\eps e_j) \nabla_i^*(a_i(y)\nabla_i\phi_{e_j}(y))=-\frac{1}{\eps}\sum_{j=1}^d \nabla_{\eps,j}u_0(x-\eps e_j)\nabla^* a(y)e_j,
\]
which cancels the second term on the r.h.s. of \eqref{eq:nau0}.

Therefore, we have 
\[
\begin{aligned}
&\nabla_\eps^* a(y)\nabla_\eps (u_0(x)+\eps u_1(x,y))-\nabla_\eps^*a_{\hom}\nabla_\eps u_0(x)\\
&=\sum_{i=1}^d (a_i(y)-\bar{a})\nabla_{\eps,i}^*\nabla_{\eps,i} u_0(x)\\
&+\sum_{i,j=1}^d \eps \nabla_{\eps,i}^*\nabla_{\eps,i}\nabla_{\eps,j}u_0(x-\eps e_j)a_i(y)\phi_{e_j}(y)\\
&+\sum_{i,j=1}^d \nabla_{\eps,i}\nabla_{\eps,j}u_0(x-\eps e_j-\eps e_i)\nabla_i^*(a_i(y)\phi_{e_j}(y))\\
&+\sum_{i,j=1}^d \nabla_{\eps,i}^* \nabla_{\eps,j}u_0(x-\eps e_j+\eps e_i)a_i(y)\nabla_i\phi_{e_j}(y):=J_1+J_2+J_3,
\end{aligned}
\]
where 
\[
J_1=\sum_{i,j=1}^d \nabla_{\eps,i}^*\nabla_{\eps,j}u_0(x-\eps e_j+\eps e_i)[a_i(y)(1_{i=j}+\nabla_i\phi_{e_j}(y))-\bar{a}1_{i=j}],
\]
\[
J_2=\sum_{i,j=1}^d \nabla_{\eps,i}\nabla_{\eps,j}u_0(x-\eps e_j-\eps e_i)\nabla_i^*(a_i(y)\phi_{e_j}(y)),
\]
and
\[
J_3=\sum_{i,j=1}^d \eps \nabla_{\eps,i}^*\nabla_{\eps,i}\nabla_{\eps,j}u_0(x-\eps e_j)a_i(y)\phi_{e_j}(y).
\]
Since $\la a(\mathrm{I}_d+\nabla\phi)\ra=\bar{a} \mathrm{I}_d$ and $\phi_{e_j}$ is stationary, we have $\la J_1\ra=\la J_2\ra=0$. 

On one hand, we define the second correctors $\psi_{2,i,j}$ and $\tilde{\psi}_{2,i,j}$ by
\begin{equation}
\nabla^* a\nabla\psi_{2,i,j}=-[a_i(1_{i=j}+\nabla_i\phi_{e_j})-\bar{a}1_{i=j}]
\label{eq:2ndcor1}
\end{equation}
and
\begin{equation}
\nabla^* a\nabla\tilde{\psi}_{2,i,j}=-\nabla_i^*(a_i\phi_{e_j}),
\label{eq:2ndcor2}
\end{equation}
then $u_2(x,y)$ is given by 
\[
\begin{aligned}
u_2(x,y)=
&\sum_{i,j=1}^d \nabla_{\eps,i}^*\nabla_{\eps,j} u_0(x-\eps e_j+\eps e_i)\psi_{2,i,j}(y)\\
+&\sum_{i,j=1}^d \nabla_{\eps,i}\nabla_{\eps,j}u_0(x-\eps e_j-\eps e_i) \tilde{\psi}_{2,i,j}(y).
\end{aligned}
\]

On the other hand, by the discussion in Section~\ref{s:q2}, we can write 
\[
a_i(y)(1_{i=j}+\nabla_i\phi_{e_j}(y))-\bar{a}1_{i=j}=\nabla^* \sigma_{ij}(y)
\]
for some stationary random field $\sigma_{ij}(y)$ with $\sigma_{ij}(0)\in L^2(\Omega)$ (this was first proved in \cite{gloria2014regularity,gu-mourratmms}). Thus $J_1+J_2$ is a linear combination of terms of the form $F_V(x)V(y)$, where $F_V$ is a second derivative of $u_0$, $V=\nabla_k^* \mathcal{V}$ for some $k=1,\ldots,d$ and stationary random field $\V(y)$ with $\V(0)\in L^2(\Omega)$. Thus we can write the equation for the second order corrector corresponding to $V$ as
\begin{equation}
\nabla^*a(y)\nabla \psi_{2,V}(y)=-V(y)=-\nabla_k^* \V(y).
\label{eq:u2}
\end{equation}
By \cite[Theorem 3, Lemma 5]{kunneman}, we can construct solutions to \eqref{eq:u2} such that 
\begin{itemize}
\item $\psi_{2,V}(0)=0$,
\item $\nabla\psi_{2,V}(y)$ is a zero-mean stationary random field and $\nabla\psi_{2,V}(0)\in L^2(\Omega)$,
\item $\psi_{2,V}(y)$ grows sublinearly in the following sense:
\begin{equation}
\la |\eps \psi_{2,V}(\frac{x}{\eps})|^2\ra \to 0
\label{eq:sublinear}
\end{equation}
as $\eps\to 0$ for all $x\in \R^d$, where we extend $\psi_2$ from $\Z^d$ to $\R^d$ such that $\psi_{2,V}(x)=\psi_{2,V}([x])$ with $[x]$ the integer part of $x\in\R^d$. Furthermore, we have 
\begin{equation}
\la |\psi_{2,V}(x)|^2 \ra \les 1+|x|^2.
\label{eq:subes}
\end{equation}

\end{itemize}


To simplify the notations, we write 
\[
J_1+J_2=\sum_VF_V(x)V(y),
\]
where the summation is over 
\[
V\in \{a_i(1_{i=j}+\nabla_i\phi_{e_j})-\bar{a}1_{i=j}, \nabla_i^*(a_i\phi_{e_j})\}_{i,j=1,\ldots,d},
\]
and write
\begin{equation}
u_2(x,y)=\sum_V F_V(x)\psi_{2,V}(y).
\label{eq:u2ex}
\end{equation}

By a similar calculation as before, we have
\begin{equation}
\begin{aligned}
&\nabla_\eps^* a(y)\nabla_\eps  (\eps^2\sum_VF_V(x)\psi_{2,V}(y))\\
&=\eps^2\sum_V \sum_{i=1}^d \nabla_{\eps,i}^* \nabla_{\eps,i}F_V(x) a_i(y)\psi_{2,V}(y)\\
&+\eps \sum_V\sum_{i=1}^d\nabla_{\eps,i} F_V(x-\eps e_i)\nabla_i^*(a_i(y)\psi_{2,V}(y))\\
&+\eps \sum_V\sum_{i=1}^d \nabla_{\eps,i}^* F_V(x+\eps e_i)a_i(y)\nabla_i\psi_{2,V}(y)\\
&+\sum_V\sum_{i=1}^d  F_V(x) \nabla_i^*(a_i(y)\nabla_i \psi_{2,V}(y)):=K_1+K_2+K_3+K_4.
\end{aligned}
\label{eq:boca}
\end{equation}
By \eqref{eq:u2}, we have $J_1+J_2+ K_4=0$, which implies
\[
\begin{aligned}
&\nabla_\eps^* a(y)\nabla_\eps (u_0(x)+\eps u_1(x,y)+\eps^2 u_2(x,y))-\nabla_\eps^*a_{\hom}\nabla_\eps u_0(x)\\
=&J_3+K_1+K_2+K_3.
\end{aligned}
\]

To summarize, with $u_1,u_2$ given by \eqref{eq:u1} and \eqref{eq:u2ex}, we can write the equation satisfied by $z_\eps=u_\eps-u_0-\eps u_1-\eps^2u_2$ as
\begin{equation}
\begin{aligned}
&(\alpha+\nabla_\eps^*a(\frac{x}{\eps})\nabla_\eps)z_\eps\\
=&f-\alpha u_0-\eps\alpha u_1-\eps^2\alpha u_2-\nabla_\eps^*a(\frac{x}{\eps})\nabla_\eps(u_0+\eps u_1+\eps^2u_2)\\
=&-\eps\alpha u_1-\eps^2\alpha u_2-J_3-K_1-K_2-K_3,
\end{aligned}
\label{eq:zeps}
\end{equation}
where we recall
\begin{eqnarray}
J_3&=&\eps\sum_{i,j=1}^d  \nabla_{\eps,i}^*\nabla_{\eps,i}\nabla_{\eps,j}u_0(x-\eps e_j)a_i(y)\phi_{e_j}(y),\\
K_1&=&\eps^2\sum_V \sum_{i=1}^d \nabla_{\eps,i}^* \nabla_{\eps,i}F_V(x) a_i(y)\psi_{2,V}(y),\\
K_2&=&\eps \sum_V\sum_{i=1}^d\nabla_{\eps,i} F_V(x-\eps e_i)\nabla_i^*(a_i(y)\psi_{2,V}(y)),\\
K_3&=&\eps \sum_V\sum_{i=1}^d \nabla_{\eps,i}^* F_V(x+\eps e_i)a_i(y)\nabla_i\psi_{2,V}(y).
\end{eqnarray}


\subsection{Proof of Theorem~\ref{t:2scale}}
\label{s:p2scale}

The goal of this section is to analyze $z_\eps$ through \eqref{eq:zeps}. Since the equation is linear, we can deal with different sources term separately. We write 
\[
(\alpha+\nabla_\eps^* a(\frac{x}{\eps})\nabla_\eps) w_\eps =W,
\]
with $W$ belonging to one of the following four groups.

\begin{itemize}
\item
\emph{Group I.} $W(x,y)=\eps^2 F_1(x)\psi_{2,V}(y)$. We have $\eps^2\alpha u_2,K_1$ belonging to this group, and the key feature is the factor of $\eps^2 \psi_{2,V}(y)$.
\item
\emph{Group II.} $W(x,y)=\eps F_1(x)\nabla_i^*(a_i(y)\psi_{2,V}(y))$, which includes $K_2$.
\item
\emph{Group III.} $W(x,y)=\eps F_1(x)F_2(y)$ with $F_2$ some stationary zero-mean random field and $F_2(0)\in L^2(\Omega)$. We have $\eps\alpha u_1$, $J_3-\la J_3\ra$ and $K_3-\la K_3\ra$ belonging to group III.
\item
\emph{Group IV.} $W(x,y)=\eps F_1(x)$ for some deterministic function $F_1$, which includes $\la J_3\ra$ and $\la K_3\ra$.
\end{itemize}

Recalling that $F_V$ are second order derivatives of $u_0$ taking the form of 
\[
\nabla_{\eps,i}^*\nabla_{\eps,j} u_0(x-\eps e_j+\eps e_i) \mbox{ or } \nabla_{\eps,i}\nabla_{\eps,j}u_0(x-\eps e_j-\eps e_i),
\]
by Lemma~\ref{l:de}, the function $F_1(x)$ appearing above in all groups can be bounded by $e^{-c|x|}$ for some $c>0$.


\subsubsection{Group I} We consider $w_\eps$ satisfying 
\[
(\alpha+\nabla_\eps^*a(\frac{x}{\eps})\nabla_\eps)w_\eps=\eps^2 F_1(x)\psi_{2,V}(y),
\]
with $|F_1(x)|^2\les e^{-c|x|}$, and a simple energy estimate together Cauchy-Schwartz inequality gives
\[
\begin{aligned}
\sum_{x\in\eps\Z^d} \la |w_\eps(x)|^2\ra \les& \eps^2 \sum_{x\in\eps\Z^d} \la |F_1(x)\psi_{2,V}(\frac{x}{\eps}) w_\eps(x)|\ra\\
\les &\eps^2 \sqrt{ \sum_{x\in\eps\Z^d} \la |F_1(x)\psi_{2,V}(\frac{x}{\eps})|^2\ra}\sqrt{\sum_{x\in\eps\Z^d}\la |w_\eps(x)|^2\ra},
\end{aligned}
\]
so we have
\[
\begin{aligned}
\eps^d \sum_{x\in\eps\Z^d}\la |w_\eps(x)|^2\ra\les& \eps^2\eps^d \sum_{x\in\eps\Z^d} \la |F_1(x)\eps\psi_{2,V}(\frac{x}{\eps})|^2\ra\\
\les &\eps^2\eps^d\sum_{x\in\eps\Z^d}  e^{-c|x|} \la |\eps \psi_{2,V}(\frac{x}{\eps})|^2\ra=\eps^2o(1),
\end{aligned}
\]
where the last step comes from \eqref{eq:sublinear} and \eqref{eq:subes}. Thus we have $\|w_\eps\|_{2,\eps}=o(\eps)$.

\subsubsection{Group II} We consider $w_\eps$ satisfying 
\[
(\alpha+\nabla_\eps^*a(\frac{x}{\eps})\nabla_\eps)w_\eps=\eps F_1(x)\nabla_i^*(a_i(y)\psi_{2,V}(y)),
\]
and use energy estimate to obtain
\[
\begin{aligned}
\sum_{x\in\eps\Z^d} ( \la |w_\eps(x)|^2\ra+ \la |\nabla_\eps w_\eps|^2\ra)\les& \eps |\la \sum_{x\in\eps\Z^d} F_1(x)\nabla_i^*(a_i(y)\psi_{2,V}(y))w_\eps(x)\ra |\\
=&\eps^2 |\la \sum_{x\in\eps\Z^d} F_1(x)\nabla_{\eps,i}^*(a_i(\frac{x}{\eps})\psi_{2,V}(\frac{x}{\eps}))w_\eps(x) \ra|.
\end{aligned}
\]
An integration by parts leads to 
\[
\begin{aligned}
\sum_{x\in\eps\Z^d} ( \la |w_\eps(x)|^2\ra+ \la |\nabla_\eps w_\eps|^2\ra)\les &\eps^2\sum_{x\in\eps\Z^d}\la |\nabla_{\eps,i}(F_1(x)w_\eps(x))a_i(\frac{x}{\eps})\psi_{2,V}(\frac{x}{\eps})|\ra\\
\les &\eps^2 \sum_{x\in\eps\Z^d} \la  |\nabla_{\eps,i}F_1(x) w_\eps(x)a_i(\frac{x}{\eps})\psi_{2,V}(\frac{x}{\eps})|\ra\\
+& \eps^2 \sum_{x\in\eps\Z^d} \la  |F_1(x+\eps e_i) \nabla_{\eps,i}w_\eps(x)a_i(\frac{x}{\eps})\psi_{2,V}(\frac{x}{\eps})|\ra.
\end{aligned}
\]
By Lemma~\ref{l:de}, $|\nabla_{\eps,i}F_1(x)|^2+|F_1(x+\eps e_i)|^2\les e^{-c|x|}$, so by Cauchy-Schwartz inequality we have
\[
\begin{aligned}
&\sum_{x\in\eps\Z^d} ( \la |w_\eps(x)|^2\ra+ \la |\nabla_\eps w_\eps|^2\ra)\\
\les& \eps^2\sqrt{\sum_{x\in\eps\Z^d}e^{-c|x|}\la |\psi_{2,V}(\frac{x}{\eps})|^2\ra} \left(\sqrt{\sum_{x\in\eps\Z^d}\la |w_\eps(x)|^2\ra}+\sqrt{\sum_{x\in\eps\Z^d}\la |\nabla_{\eps,i}w_\eps(x)|^2\ra}
\right)\\
\les &\eps^2\sqrt{\sum_{x\in\eps\Z^d}e^{-c|x|}\la |\psi_{2,V}(\frac{x}{\eps})|^2\ra} \sqrt{\sum_{x\in\eps\Z^d} ( \la |w_\eps(x)|^2\ra+ \la |\nabla_\eps w_\eps|^2\ra)},
\end{aligned}
\]
thus
\[
\eps^d\sum_{x\in\eps\Z^d} ( \la |w_\eps(x)|^2\ra+ \la |\nabla_\eps w_\eps|^2\ra)\les \eps^2 \eps^d\sum_{x\in\eps\Z^d}e^{-c|x|}\la |\eps\psi_{2,V}(\frac{x}{\eps})|^2\ra=\eps^2o(1),
\]
and we conclude $
\|w_\eps\|_{2,\eps}=o(\eps)$.

\subsubsection{Group III} We consider $z_\eps$ satisfying 
\[
(\alpha+\nabla_\eps^*a(\frac{x}{\eps})\nabla_\eps)w_\eps=\eps F_1(x)F_2(y),
\]
where $|F_1(x)|\les e^{-c|x|}$ and $F_2$ is a zero-mean stationary random field with $F_2(0)\in L^2(\Omega)$. 
The proof borrows the following result from \cite[Equation (4.23i) -- (4.23iii)]{kunneman}, which itself comes from \cite{papanicolaou1979boundary} in the continuous setting: there exists a random field $H(y)=(H_1(y),\ldots,H_d(y))$ on $\Z^d$ that satisfies the following properties:
\begin{itemize}
\item $H_j(0)=0$ and $\sum_{j=1}^d \nabla_j H_j(y)=F_2(y), y\in\Z^d$,
\item by extending $H(y)$ from $\Z^d$ to $\R^d$ such that $H(x)=H([x])$ for any $x\in\R^d$ with $[x]$ the integer part of $x$, we have
\begin{equation}
\la |\eps H(\frac{x}{\eps})|^2\ra\to 0
\label{eq:sub1}
\end{equation}
as $\eps\to 0$, and 
\begin{equation}
\la |H(x)|^2\ra \les 1+|x|^2.
\label{eq:sub2}
\end{equation}
\end{itemize}

Now we only need to apply the same proof as for \emph{Group II} to conclude 
\[
\|w_\eps\|_{2,\eps}=o(\eps).
\]

\subsubsection{Group IV} This is the key part where we show the deterministic bias vanishes in the order of $\eps$. We consider $w_\eps$ satisfying 
\[
(\alpha+\nabla_\eps^*a(\frac{x}{\eps})\nabla_\eps)w_\eps=\eps F_1(x),
\]
where $F_1:\eps\Z^d\to \R$ is some fast decaying deterministic function. If we define $z_0$ satisfying 
\[
(\alpha+\nabla_\eps^*a_{\hom}\nabla_\eps)w_0=\eps F_1(x),
\]
it is clear by the standard homogenization result that $
\|w_\eps-w_0\|_{2,\eps} =o(\eps)$.

Now we claim that in our specific case the deterministic bias $\|w_0\|_{2,\eps}=o(\eps)$, which implies $\|w_\eps\|_{2,\eps}=o(\eps)$.

First we note that 
\begin{equation}
\begin{aligned}
(\alpha+\nabla_\eps^* a_{\hom}\nabla_\eps) w_0=&-\la J_3\ra-\la K_3\ra\\
=&-\eps\sum_{i,j=1}^d  \nabla_{\eps,i}^*\nabla_{\eps,i}\nabla_{\eps,j}u_0(x-\eps e_j)\la a_i\phi_{e_j}\ra\\
&-\eps\sum_V\sum_{k=1}^d \nabla_{\eps,k}^* F_V(x+\eps e_k)\la a_k\nabla_k\psi_{2,V}\ra,
\end{aligned}
\label{eq:w0}
\end{equation}
and by recalling $u_2(x,y)$ given by 
\[
\begin{aligned}
u_2(x,y)=\sum_VF_V(x)\psi_{2,V}(y)=&\sum_{i,j=1}^d \nabla_{\eps,i}^*\nabla_{\eps,j} u_0(x-\eps e_j+\eps e_i)\psi_{2,i,j}(y)\\
+&\sum_{i,j=1}^d \nabla_{\eps,i}\nabla_{\eps,j}u_0(x-\eps e_j-\eps e_i) \tilde{\psi}_{2,i,j}(y),
\end{aligned}
\]
we can write the second term on the r.h.s. of \eqref{eq:w0} as 
\[
\begin{aligned}
&\sum_V\sum_{k=1}^d \nabla_{\eps,k}^* F_V(x+\eps e_k)\la a_k\nabla_k\psi_{2,V}\ra\\
=&\sum_{k,i,j=1}^d \nabla_{\eps,k}^*\nabla_{\eps,i}^*\nabla_{\eps,j} u_0(x-\eps e_j+\eps e_i+\eps e_k) \la a_k\nabla_k\psi_{2,i,j}\ra\\
+&\sum_{k,i,j=1}^d \nabla_{\eps,k}^*\nabla_{\eps,i}\nabla_{\eps,j}u_0(x-\eps e_j-\eps e_i+\eps e_k)\la a_k\nabla_k\tilde{\psi}_{2,i,j}\ra.
\end{aligned}
\]

By Lemma~\ref{l:fide}, we have $\|w_0-\bar{w}_0\|_{2,\eps}=o(\eps)$ with $\bar{w}_0$ solving the corresponding equation in $\R^d$:
\begin{equation}
\begin{aligned}
\frac{1}{\eps}(\alpha-\nabla\cdot a_{\hom}\nabla)\bar{w}_0
&= \sum_{i,j=1}^d \partial_{x_ix_ix_j} \bar{u}_0(x)\la a_i\phi_{e_j}\ra\\
&- \sum_{k,i,j=1}^d \partial_{x_kx_ix_j}\bar{u}_0(x)\la a_k\nabla_k\psi_{2,i,j}\ra+ \sum_{k,i,j=1}^d \partial_{x_kx_ix_j}\bar{u}_0(x)\la a_k\nabla_k\tilde{\psi}_{2,i,j}\ra,
\end{aligned}
\label{eq:z0}
\end{equation}
where $\bar{u}_0$ satisfies 
\[
(\alpha-\nabla\cdot a_{\hom} \nabla)\bar{u}_0=f
\]
on $\R^d$. Now we show that r.h.s. of \eqref{eq:z0} is zero so $\bar{w}_0\equiv 0$. 

For the second term on the r.h.s. of \eqref{eq:z0}, we have
\[
\la a_k\nabla_k\psi_{2,i,j}\ra=\la \nabla_k^* a_k\psi_{2,i,j}\ra=-\la \nabla^* a\nabla \phi_{e_k}\psi_{2,i,j}\ra=-\la \phi_{e_k}\nabla^*a\nabla\psi_{2,i,j}\ra.
\]
By the equation satisfied by $\psi_{2,i,j}$ \eqref{eq:2ndcor1}, we have
\[
\la a_k\nabla_k\psi_{2,i,j}\ra=\la \phi_{e_k}a_i(1_{i=j}+\nabla_i\phi_{e_j})\ra=\la \phi_{e_k}a_i\ra1_{i=j}+\la \phi_{e_k}a_i\nabla_i\phi_{e_j}\ra.
\]

Similarly, for the third term on the r.h.s. of \eqref{eq:z0}, we have
\[
\la a_k\nabla_k\tilde{\psi}_{2,i,j}\ra=-\la\phi_{e_k}\nabla^*a\nabla\tilde{\psi}_{2,i,j}\ra=\la\phi_{e_k}\nabla_i^*(a_i\phi_{e_j})\ra=\la \nabla_i\phi_{e_k}a_i\phi_{e_j}\ra.
\] 
Therefore, we have 
\[
\begin{aligned}
\mbox{r.h.s. of \eqref{eq:z0}}=&\sum_{i,j=1}^d \partial_{x_ix_ix_j}\bar{u}_0(x)\la a_i\phi_{e_j}\ra-\sum_{k,i,j=1}^d \partial_{x_kx_ix_j}\bar{u}_0(x)\la\phi_{e_k} a_i\ra 1_{i=j}\\
-&\sum_{k,i,j=1}^d \partial_{x_kx_ix_j} \bar{u}_0(x)\la  \phi_{e_k}a_i\nabla_i\phi_{e_j}\ra+\sum_{k,i,j=1}^d \partial_{x_kx_ix_j}\bar{u}_0(x)\la \nabla_i\phi_{e_k}a_i\phi_{e_j}\ra\equiv0.
\end{aligned}
\]

To summarize, we have $\|w_\eps\|_{2,\eps}=o(\eps)$.

\subsubsection{Summary} 

Recall that $z_\eps(x)=u_\eps(x)-u_0(x)-\eps u_1(x,y)-\eps^2 u_2(x,y)$ and 
\[
\begin{aligned}
&(\alpha+\nabla_\eps^*a(\frac{x}{\eps})\nabla_\eps)z_\eps\\
=&-\eps\alpha u_1-\eps^2\alpha u_2-(J_3-\la J_3\ra)-K_1-K_2-(K_3-\la K_3\ra)-\la J_3\ra-\la K_3\ra.
\end{aligned}
\]

By the previous discussion on \emph{Group I,II,III,IV}, we have $\|z_\eps\|_{2,\eps}=o(\eps)$ and the discussion for \emph{Group I} already shows that
\[
\|\eps^2u_2\|_{2,\eps}=o(\eps),
\]
and this gives 
\[
\|u_\eps-u_0-\eps u_1\|_{2,\eps}=o(\eps).
\]
Now we only need to note that $\|u_1-\tilde{u}_1\|_{2,\eps}=O(\eps)$ with
\[
\tilde{u}_1(x,\frac{x}{\eps})=\sum_{j=1}^d \nabla_{\eps,j}u_0(x)\phi_{e_j}(\frac{x}{\eps})
\]
to complete the proof of Theorem~\ref{t:2scale}.

\section{Higher order fluctuations}
\label{s:high}

With the proof of Theorem~\ref{t:2scale}, it is clear that a similar approach will lead to a higher order expansions of $u_\eps$ in higher dimensions, e.g., when $d\geq 5$, we can include in the expansion the second order correctors (which are stationary by Theorem~\ref{t:mainTh}) and prove the remainder is of order $o(\eps^2)$. The calculation is more involved but the idea is the same. We first illustrate it in the continuous setting, then we present a proof of Theorem~\ref{t:highex} in the discrete setting.

\subsection{High order errors in the continuous setting}
\label{s:excon}

For fixed $n\geq 2$ and equations
\[
(\alpha-\nabla\cdot a(\frac{x}{\eps})\nabla) u_\eps(x)= f(x),\ \
(\alpha-\nabla\cdot a_{\hom}\nabla) u_0(x)=f(x),
\]
with $x\in\R^d,\alpha>0,f\in\C_c^\infty(\R^d)$, the goal is to obtain an expansion of $u_\eps-u_0$ up to order $\eps^n$ in dimension $d\geq 2n+1$. 

When $d\geq 2n+1$, the $n-$th order correctors are stationary by Theorem~\ref{t:mainTh}. To show the remainder is small, we need to construct up to $(n+1)$-th order correctors. It can be done as in \cite[Theorem 2]{papanicolaou1979boundary} since we have zero-mean stationary gradients.  

Now we consider 
\[
\mathscr{J}_n:=[\alpha-(\nabla_x+\frac{1}{\eps}\nabla_y)\cdot a(y) (\nabla_x+\frac{1}{\eps}\nabla_y)] (u_0+\eps u_1+\ldots+\eps^n u_n+\eps^{n+1}u_{n+1}), 
\]
and 
 construct $u_1,u_2, \{u_k\}_{3\leq k\leq n+1}$ satisfy the following equations:
\[
\nabla_y\cdot a(y)\nabla_y u_1+\nabla_y\cdot a(y)\nabla_x u_0=0,
\]
\[
\begin{aligned}
&\nabla_y\cdot a(y)\nabla_yu_2+\nabla_y\cdot a(y)\nabla_x u_1+\nabla_x\cdot a(y)\nabla_y u_1+\nabla_x \cdot a(y)\nabla_x u_0\\
=&\la \nabla_x\cdot a(y)\nabla_y u_1+\nabla_x \cdot a(y)\nabla_x u_0\ra,
\end{aligned}
\]
and
\begin{equation}
\begin{aligned}
&\nabla_y \cdot a(y)\nabla_y u_k+\nabla_y\cdot a(y)\nabla_x u_{k-1}+\nabla_x \cdot a(y)\nabla_y u_{k-1}+\nabla_x \cdot a(y)\nabla_x u_{k-2}-\alpha u_{k-2}\\
=&\la \nabla_x \cdot a(y)\nabla_y u_{k-1}+\nabla_x \cdot a(y)\nabla_x u_{k-2}\ra.
\end{aligned}
\label{eq:highcor1}
\end{equation}

\begin{rem}
Our definition of higher order correctors $u_k,k\geq 3$ in \eqref{eq:highcor1} is different from \eqref{eq:highcorEq} due to the extra term $\alpha u_{k-2}$, but it is clear that Theorem~\ref{t:mainTh} still applies in the corresponding discrete setting. 
\end{rem}

By the above construction, it is clear that
\begin{equation}
\begin{aligned}
\mathscr{J}_n=&\alpha u_0-\sum_{k=2}^{n+1}\eps^{k-2}\la \nabla_x \cdot a(y)\nabla_y u_{k-1}+\nabla_x \cdot a(y)\nabla_x u_{k-2}\ra\\
-&\eps^n\nabla_x\cdot a(y)\nabla_x u_n-\eps^{n}\nabla_x\cdot a(y)\nabla_y u_{n+1}-\eps^{n}\nabla_y \cdot a(y)\nabla_x u_{n+1}-\eps^{n+1}\nabla_x\cdot a(y)\nabla_x u_{n+1}\\
+& \alpha \eps^n u_n+\alpha \eps^{n+1} u_{n+1}.
\end{aligned}
\end{equation}

By defining 
\begin{equation}
z_\eps=u_\eps-u_0-\eps u_1-\ldots-\eps^{n+1}u_{n+1},
\label{eq:ex1}
\end{equation}
and the fact that $\alpha u_0-\la \nabla_x \cdot a(y)\nabla_y u_1+\nabla_x \cdot a(y)\nabla_x u_0\ra=f$, we derive
\begin{equation}
\begin{aligned}
(\alpha-\nabla\cdot a(\frac{x}{\eps})\nabla)z_\eps=&f-\mathscr{J}_n\\
=
&\sum_{k=3}^{n+2}\eps^{k-2}\la \nabla_x \cdot a(y)\nabla_y u_{k-1}+\nabla_x \cdot a(y)\nabla_x u_{k-2}\ra\\
&+\eps^n\nabla_x\cdot a(y)\nabla_x u_n-\la\eps^n\nabla_x\cdot a(y)\nabla_x u_n\ra\\
&+\eps^{n}\nabla_x\cdot a(y)\nabla_y u_{n+1}-\la\eps^{n}\nabla_x\cdot a(y)\nabla_y u_{n+1}\ra\\
&+\eps^{n}\nabla_y \cdot a(y)\nabla_x u_{n+1}+\eps^{n+1}\nabla_x\cdot a(y)\nabla_x u_{n+1}\\
&-\alpha \eps^n u_n-\alpha \eps^{n+1} u_{n+1},
\label{eq:exhigh}
\end{aligned}
\end{equation}
and the goal reduces to refine $z_\eps$ up to the order $\eps^n$. 

There are two types of sources on the r.h.s. of \eqref{eq:exhigh}:
\begin{itemize}
\item \emph{random sources}: by the same proof for \emph{Group I,II,III} in Section~\ref{s:p2scale}, we have the terms in the last four lines of the above display contribute $o(\eps^n)$ to $z_\eps$. Here we need a stationary $n-$th order corrector, and a sublinear $(n+1)$-th order corrector with a zero-mean stationary gradient.
\item \emph{deterministic sources}: we can write
\begin{equation}
\sum_{k=3}^{n+2}\eps^{k-2}\la \nabla_x \cdot a(y)\nabla_y u_{k-1}+\nabla_x \cdot a(y)\nabla_x u_{k-2}\ra=\sum_{k=1}^n \eps^k f_k(x)
\label{eq:desou}
\end{equation}
where $f_k$ is some linear combination of derivatives of $u_0$. For equations of the form
\[
(\alpha-\nabla\cdot a(\frac{x}{\eps})\nabla)w_\eps(x)=f_k(x),
\]
we can apply the expansion again, and derive equations of the form \eqref{eq:exhigh} to obtain an expansion of $w_\eps$. By iteration, we can go up to the order of $\eps^n$ in finite steps.
\end{itemize}

To summarize, when $d\geq 2n+1$, we expect there exists $\{u_k\}_{k=1}^n$ and $\{v_k\}_{k=1}^n$ such that $\la u_i\ra=0$, $v_i$ are deterministic, and
\begin{equation}
\left(\int_{\R^d}\bigg\la \bigg|u_\eps-u_0-\sum_{k=1}^n \eps^k u_k-\sum_{k=1}^n \eps^kv_k\bigg|^2\bigg\ra \, \d x\right)^{\frac12}=o(\eps^n).
\label{eq:highexre}
\end{equation}

It is worth mentioning that the $\{u_k\}_{k=1}^n$ appearing in \eqref{eq:highexre} is different from the ones in \eqref{eq:ex1}, since the deterministic source in \eqref{eq:desou} also contributes to the random error.

\subsection{Proof of Theorem~\ref{t:highex}}

For any $n\geq 2$, to expand up to $o(\eps^n)$, we need to construct the $(n+1)$-th correctors $\psi_{n+1}$. By Theorem~\ref{t:mainTh} and \cite[Theorem 3, Lemma 5]{kunneman}, this can be done when $d\geq 2n+1$, and we have a sublinear random field $\psi_{n+1}(y)$ with $\psi_{n+1}(0)=0$ and a zero-mean stationary gradient $\nabla\psi_{n+1}$. We can continue the calculation in Section~\ref{s:conu1u2} and construct higher order correctors $\{u_k\}_{3\leq k\leq n+1}$ as in Section~\ref{s:excon}. In the following we only discuss the case $n=2$ since it already includes all the ingredients of the proof. 

To construct the third order correctors, we recall \eqref{eq:zeps}
\begin{equation}
\begin{aligned}
&(\alpha+\nabla_\eps^*a(\frac{x}{\eps})\nabla_\eps)(u_\eps-u_0-\eps u_1-\eps^2u_2)\\
=&-\eps\alpha u_1-\eps^2\alpha u_2-(J_3-\la J_3\ra)-K_1-K_2-(K_3-\la K_3\ra)-\la J_3\ra-\la K_3\ra,
\label{eq:ex2nd}
\end{aligned}
\end{equation}
with 
\begin{eqnarray*}
J_3&=&\eps\sum_{i,j=1}^d  \nabla_{\eps,i}^*\nabla_{\eps,i}\nabla_{\eps,j}u_0(x-\eps e_j)a_i(y)\phi_{e_j}(y),\\
K_1&=&\eps^2\sum_V \sum_{i=1}^d \nabla_{\eps,i}^* \nabla_{\eps,i}F_V(x) a_i(y)\psi_{2,V}(y),\\
K_2&=&\eps \sum_V\sum_{i=1}^d\nabla_{\eps,i} F_V(x-\eps e_i)\nabla_i^*(a_i(y)\psi_{2,V}(y)),\\
K_3&=&\eps \sum_V\sum_{i=1}^d \nabla_{\eps,i}^* F_V(x+\eps e_i)a_i(y)\nabla_i\psi_{2,V}(y).
\end{eqnarray*}
Similar to $u_2$, we define $u_3$ to get rid of the zero-mean terms on the r.h.s. of \eqref{eq:ex2nd} of order $\eps$, which is written as  
\[
\eps\sum_U G_U(x)U(y)=-\eps \alpha u_1 -(J_3-\la J_3\ra) -K_2-(K_3-\la K_3\ra).
\]
Here $G_U$ is some derivative of $u_0$ (note that due to the term $-\eps \alpha u_1$, $G_U$ is not necessarily a \emph{third order} derivative of $u_0$) and the summation is over zero-mean stationary random fields
 \[
 U\in \{\phi_{e_i},a_i\phi_{e_j}-\la a_i\phi_{e_j}\ra,\nabla_i^*(a_i\psi_{2,V}),a_i\nabla_i\psi_{2,V}-\la a_i\nabla\psi_{2,V}\ra\}_{i,j=1,\ldots,d}.
\]

We define
\[
u_3(x,y)=\sum_U  G_U(x)\psi_{3,U}(y),
\]
with $\psi_{3,U}$ the corrector corresponding to $U$, i.e., 
\[
\nabla^*a(y)\nabla \psi_{3,U}(y)=U(y).
\]
 By the same calculation as in \eqref{eq:boca}, we have
\[
\begin{aligned}
\nabla_\eps^* a(y)\nabla_\eps  (\eps^3u_3)=&\eps^3\sum_U \sum_{i=1}^d \nabla_{\eps,i}^* \nabla_{\eps,i}G_U(x) a_i(y)\psi_{3,U}(y)\\
+&\eps^2 \sum_U\sum_{i=1}^d\nabla_{\eps,i} G_U(x-\eps e_i)\nabla_i^*(a_i(y)\psi_{3,U}(y))\\
+&\eps^2 \sum_U\sum_{i=1}^d \nabla_{\eps,i}^* G_U(x+\eps e_i)a_i(y)\nabla_i\psi_{3,U}(y)\\
+&\eps \sum_U\sum_{i=1}^d  G_U(x) \nabla_i^*(a_i(y)\nabla_i \psi_{3,U}(y)):=L_1+L_2+L_3+L_4.
\end{aligned}
\]
It is clear that $L_4=\eps\sum_U G_U(x)U(y)$, and by \eqref{eq:ex2nd}, we have 
\[
z_\eps=u_\eps-u_0-\eps u_1-\eps^2u_2-\eps^3u_3
\]
 satisfies
\[
\begin{aligned}
&(\alpha+\nabla_\eps^*a(\frac{x}{\eps})\nabla_\eps)z_\eps\\
=&-\eps^2\alpha u_2-\eps^3\alpha u_3-K_1-\la J_3\ra-\la K_3\ra-L_1-L_2-L_3.
\end{aligned}
\]
Similar to \eqref{eq:exhigh}, we can write 
\[
\begin{aligned}
(\alpha+\nabla_\eps^*a(\frac{x}{\eps})\nabla_\eps)z_\eps=&-\la J_3\ra-\la K_3\ra-\la K_1\ra-\la L_3\ra\\
&-K_1+\la K_1\ra -L_3+\la L_3\ra\\
&-L_1-L_2\\
&-\eps^2\alpha u_2-\eps^3\alpha u_3.
\end{aligned}
\]

For the random source in the last three lines of the above display, by the same proof as in Section~\ref{s:p2scale}, their contributions to $z_\eps$ is $o(\eps^2)$. 

For the deterministic source $-\la J_3\ra-\la K_3\ra-\la K_1\ra-\la L_3\ra$, it is of the form $\eps F_1(x)$, where $F_1(x),x\in\eps\Z^d$ is deterministic and fast decaying, so we can apply the same expansion again to refine the solution to 
\[
(\alpha+\nabla_\eps^*a(\frac{x}{\eps})\nabla_\eps)w_\eps=-\la J_3\ra-\la K_3\ra-\la K_1\ra-\la L_3\ra
\]
up to $o(\eps^2)$. More precisely, for the source $-\la K_1\ra-\la L_3\ra$, it is of order $O(\eps^2)$, so we can define $v_2$ as the solution to
\[
(\alpha+\nabla_\eps^*a_{\hom}\nabla_\eps)v_2=-(\la K_1\ra+\la L_3\ra)/\eps^2.
\]
For the source $-\la J_3\ra-\la K_3\ra$, it is of order $O(\eps)$. If we define $v_\eps$ as the solution to 
\[
(\alpha+\nabla_\eps^* a(\frac{x}{\eps}) \nabla_\eps)v_\eps= -(\la J_3\ra+\la K_3\ra)/\eps,
\]
then by Theorem~\ref{t:2scale}, $v_\eps$ can be refined up to order $o(\eps)$, i.e., let $v_1$ solve
\[
(\alpha+\nabla_\eps^* a_{\hom} \nabla_\eps)v_1= -(\la J_3\ra+\la K_3\ra)/\eps,
\]
which is the same as \eqref{eq:w0}, and there exists $v_{11}$ such that $\la v_{11}\ra=0$ 
 and
\[
\|v_\eps-v_1-\eps v_{11}\|_{2,\eps}=o(\eps).
\]
Now it is clear that 
\[
\|w_\eps-\eps v_1-\eps^2v_{11}-\eps^2 v_2\|_{2,\eps}=o(\eps^2).
\]

To summarize, we have shown that 
\[
\|u_\eps-u_0-\eps u_1-\eps^2u_2-\eps v_1-\eps^2v_{11}-\eps^2v_2\|_{2,\eps}=o(\eps^2),
\]
where $u_1,u_2,v_{11}$ are zero-mean random fluctuations and $v_1,v_2$ are deterministic bias. This completes the proof of Theorem~\ref{t:highex} for $n=2$. 



\section{Discussion: local vs global or strong vs weak random fluctuations}
\label{s:compa}

Theorem~\ref{t:2scale} shows that when $d\geq 3$, $\|u_\eps-u_0-\eps u_1\|_{2,\eps}=o(\eps)$ with the first order random fluctuations represented by 
\[
\eps u_1(x,\frac{x}{\eps})=\eps \sum_{j=1}^d \nabla_{\eps,j}u_0(x)\phi_{e_j}(\frac{x}{\eps}).
\]
It is consistent with \cite[Theorem 2.5]{gu-mourrat} where a pointwise version is proved in the continuous setting: for fixed $x\in \R^d$, 
\begin{equation}
u_\eps(x)=u_0(x)+\eps \sum_{j=1}^d \partial_{x_j}u_0(x)\phi_{e_j}(\frac{x}{\eps})+o(\eps),
\label{eq:point}
\end{equation}
with $o(\eps)/\eps\to0$ in $L^1(\Omega)$.

Neither of the two results implies the other. On one hand, it is not clear how to obtain the pointwise estimates by the analytic approach, in particular the energy estimate described in Section~\ref{s:p2scale}; on the other hand, the probabilistic approach used in \cite{gu-mourrat} loses track of the dependence of $o(\eps)$ in \eqref{eq:point} on $x\in\R^d$, so \eqref{eq:point} does not easily extend to an $L^2(\R^d\times \Omega)$ version.


Both results 
indicate that  $u_1(x,x/\eps)$ represents the first order random fluctuation measured in a strong sense. It is however not the case after a spatial average with respect to a test function as pointed out in \cite{gu-mourratmms}. Central limit theorems were derived for large scale fluctuations of $\eps u_1(x,x/\eps)$ and $u_\eps-\la u_\eps\ra$ in \cite{mourrat2014correlation,MN, gu-mourratmms,gloria2016random,armstrong2016scaling}. When $d\geq 3$, the result shows for any $g\in\C_c^\infty(\R^d)$ 
\begin{equation}
\frac{1}{\eps^{d/2}}\int_{\R^d} \eps u_1(x,\frac{x}{\eps})g(x)\,\d x\Rightarrow N(0,\sigma^2),
\label{eq:cltcor}
\end{equation}
\begin{equation}
\frac{1}{\eps^{d/2}}\int_{\R^d} (u_\eps(x)-\la u_\eps(x)\ra)g(x)\,\d x \Rightarrow N(0,\tilde{\sigma}^2),
\label{eq:cltso}
\end{equation}
and $\sigma^2\neq \tilde{\sigma}^2$. The mismatch between the two variances in \eqref{eq:cltcor} and \eqref{eq:cltso} suggests that the higher order correctors become visible in \eqref{eq:cltso}. When $d\geq 2n+1$, the $n-$th order correctors $\psi_n$ are stationary, and a covariance estimate gives
\[
\begin{aligned}
|\la \psi_n(0)\psi_n(x)\ra| \les& \sum_e \|\partial_e \psi_n(0)\|_2\|\partial_e \psi_n(x)\|_2\\
\les &\sum_e \frac{1}{|\un{e}|_*^{(d-n)-}}\frac{1}{|x-\un{e}|_*^{(d-n)-}}\les \frac{1}{|x|_*^{(d-2n)-}}.
\end{aligned}
\]
The scaling indicates that as a stationary random field, $\psi_n(x)$ decorrelates almost at the rate of $|x|^{-(d-2n)}$ (it was proved rigorously in \cite{mourrat2014correlation} for $n=1$). In other words, the higher order correctors have stronger correlations and the decay of correlation functions slows down as the order increases! It partly explains why we have contributions from high order correctors weakly in space. If we abuse the notation and consider
\[
I_n:=\frac{1}{\eps^{d/2}}\int_{\R^d} \eps^n \psi_n(\frac{x}{\eps}) g(x)dx,
\]
we have
\[
\la |I_n|^2\ra =\int_{\R^{2d}} \eps^{2n-d} \la \psi_n(\frac{x}{\eps})\psi_n(\frac{y}{\eps})\ra g(x)g(y)dxdy.
\]
If we assume $\la \psi_n(0)\psi_n(x)\ra\sim |x|^{-(d-2n)}$ as $|x|\to\infty$, the scaling indicates that $\la |I_n|^2\ra$ is of order $O(1)$, i.e., the $n-$th order corrector contributes weakly in space.

In the recent preprint \cite{gloria2016random}, the authors provided another way of understanding the mismatch between \eqref{eq:cltcor} and \eqref{eq:cltso} in terms of the so-called homogenization commutator. It turns out that if the formal expansion applies to the gradient $\nabla u_\eps$, it leads to the correct large scale random fluctuation; see \cite[Page 6]{gloria2016random}. The expansion on the gradient of the Green's function in \cite[Proposition 4.2]{gu-mourratmms} shares the same spirit. It is worth mentioning that a nice non-local expansion was proved in \cite[Corollary 2]{gloria2016random} which identifies a single term that captures the large scale fluctuation.

\appendix

\section{The existence of stationary correctors: removing the massive term}
\label{s:exst}

The goal in this section is to show that for the corrector equation \eqref{eq:corndislambda} defined when $d\geq 2n-1$
\begin{equation}
(\lambda+\nabla^*a(y)\nabla) \psi_n^\lambda(y)=F(y)=\left\{\begin{array}{l}
\nabla_i^* (a_i(y)\psi_{n-1}(y)),\\
a_i(y) \nabla_i\psi_{n-1}(y)-\la a_i(y)\nabla_i \psi_{n-1}(y)\ra,\\
a_{i}(y)\psi_{n-2}(y)-\la a_{i}(y)\psi_{n-2}(y)\ra,
\end{array}
\right.
\end{equation}
the uniform estimates $\|\psi_n^\lambda\|_p\les 1$ and $\|\nabla\psi_n^\lambda\|_p\les 1$ implies the existence of a stationary corrector and a stationary gradient, respectively. 

In Section~\ref{s:q2}, we proved that when $d\geq 2n-1$, there exists a stationary random field $\Psi$ such that $\Psi(0)\in L^2(\Omega)$ and 
\[
F=\nabla^* \Psi.
\]
The rest is standard, and we present it for the convenience of the reader. 

We first introduce some notations. Let $\omega=(\omega_e)_{e\in\B}\in \Omega$ denote the sample point, and for $x\in \Z^d$, we define the shift operator $\tau_x$ on $\Omega$ by $(\tau_x \omega)_e=\omega_{x+e}$, where $x+e:=(x+\underline{e},x+\bar{e})$ is the edge obtained by shifting $e$ by $x$. Since $\{\omega_e\}_{e\in\mathbb{B}}$ are i.i.d., $\{\tau_x\}_{x\in \Z^d}$ is a group of measure-preserving transformations. We can define the operator 
\[
T_x f(\zeta)=f(\tau_x\zeta)
\]
for any measurable function $f$ on $\Omega$, and the generators of $T_x$, denoted by $\{D_i\}_{i=1}^d$, are defined by $D_if:=T_{e_i}f-f$. The adjoint $D_i^*$ is defined by $D_i^* f:=T_{-e_i}f-f$. We denote the gradient on $\Omega$ by $D=(D_1,\ldots,D_d)$ and the divergence $D^*F:=\sum_{i=1}^d D_i^*F_i$ for $F:\Omega\to \R^d$. 

\begin{lem}
Let $\phi_\lambda$ solve $(\lambda+\nabla^* a\nabla)\phi_\lambda=\nabla^*\Psi$ with $\Psi$ a stationary random field on $\Z^d$ such that $\Psi(0)\in L^2(\Omega)$. For the equation 
\begin{equation}
\nabla^*a\nabla\phi=\nabla^*\Psi,
\label{eq:corap}
\end{equation}

(i) there exists a random field $\phi$ solving \eqref{eq:corap} such that $\nabla\phi$ is stationary and $\nabla\phi(0)\in L^2(\Omega)$.

(ii) if $\la |\phi_\lambda|^2\ra \les1$, then there exists a stationary random field $\phi$ solving \eqref{eq:corap} such that $\phi(0)\in L^2(\Omega)$.
\label{l:stacor}
\end{lem}

\begin{proof}
Part (i) comes from \cite[Theorem 3]{kunneman}. For Part (ii), let $\tilde{a}(\omega)=(\omega_{e_1},\ldots,\omega_{e_d})$ and $\tilde{\Psi}\in L^2(\Omega)$ so that $\Psi(x)=\tilde{\Psi}(\tau_x\omega)$, we lift the equation to the probability space
\[
(\lambda+D^*\tilde{a} D)\tilde{\phi}_\lambda=D^*\tilde{\Psi},
\]
so it is clear that $\phi_\lambda(x)=\tilde{\phi}_\lambda(\tau_x\omega)$. Since $\la |\tilde{\phi}_\lambda|^2\ra\les 1$, we can extract a subsequence $\tilde{\phi}_\lambda\to \tilde{\phi}$ weakly in $L^2$, which implies $D\tilde{\phi}_\lambda\to D\tilde{\phi}$ weakly in $L^2(\Omega)$. 

For any $G\in L^2(\Omega)$, we have
\[
\lambda\la \tilde{\phi}_\lambda G\ra+\la \tilde{a}D\tilde{\phi}_\lambda DG\ra=\la D^*\tilde{\Psi}G\ra,
\]
and sending $\lambda\to 0$ leads to 
\[
\la \tilde{a}D\tilde{\phi}DG\ra=\la D^*\tilde{\Psi}G\ra,
\]
so $D^*\tilde{a}D\tilde{\phi}=D^*\tilde{\Psi}$. Now we define $\phi(x)=\tilde{\phi}(\tau_x\omega)$, and it is clear that 
\[
\nabla^*a\nabla\phi=\nabla^*\Psi.
\]
The proof is complete.
\end{proof}

\section{Finite difference approximation}

The following are some classical results and we present it for the convenience of the reader. We recall the classical Green's function estimates: 
\begin{equation}
\G^\alpha(x,y)\les \frac{e^{-c\sqrt{\alpha}|x-y|}}{|x-y|_*^{d-2}}
\label{eq:gres}
\end{equation}
for some $c>0$, with $\G^\alpha$ the Green's function of $\alpha+\nabla^*\nabla$ on $\Z^d$. Let $\G_\eps^\alpha(x,y)$ be the Green's function of $\alpha+\nabla_\eps^*\nabla_\eps$ on $\eps\Z^d$, it is clear that 
\[
\G_\eps^\alpha(x,y)=\eps^2\G^{\eps^2 \alpha}(\frac{x}{\eps},\frac{y}{\eps}).
\]

\begin{lem}
Let $u_\eps$ solve
\[
(\alpha+\nabla_\eps^*\nabla_\eps)u_\eps=f_\eps, 
\]
on $\eps\Z^d$, with $f_\eps$ satisfying (i) there exists $\lambda>0$ s.t. $|f_\eps (x)|\les e^{-\lambda|x|}$, (ii) all discrete derivatives of $f_\eps$ satisfies (i), then for $v_\eps=u_\eps$ or any derivative of $u_\eps$, 
 we have 
 \[
 |v_\eps (x)|\les e^{-c|x|}
 \]
 for constant $c>0$ independent of $\eps>0,x\in\eps\Z^d$.
 \label{l:de}
\end{lem}

\begin{proof}
It suffices to consider the case $v_\eps=u_\eps$. For derivatives, e.g., $v_\eps=\nabla_{\eps,i} u_\eps$, using the fact that $\nabla_{\eps,i}$ commutes with $\nabla_{\eps,j}$ and $\nabla_{\eps,j}^*$, we have
\[
(\alpha+\nabla_\eps^*\nabla_\eps)v_\eps=\nabla_{\eps,i}f_\eps,
\]
so we only need to apply the result when $v_\eps=u_\eps$. 

For $u_\eps(x)$, we use the Green's function representation
\[
u_\eps(x)=\sum_{y\in\eps\Z^d} \G^\alpha_\eps(x,y)f(y)=\sum_{y\in\eps\Z^d}\eps^2 \G^{\eps^2\alpha}(\frac{x}{\eps},\frac{y}{\eps})f(y),
\]
and by \eqref{eq:gres}, we have
\[
\eps^2 \G^{\eps^2\alpha}(\frac{x}{\eps},\frac{y}{\eps})\les \eps^2 \frac{e^{-\rho|x-y|}}{|\frac{x-y}{\eps}|_*^{d-2}}
\]
for some constant $\rho>0$. Thus 
\[
|u_\eps(x)|\les \eps^d\sum_{y\in \eps \Z^d, y\neq x} \frac{e^{-\rho |x-y|}}{|x-y|^{d-2}}e^{-\lambda|y|}+\eps^2e^{-\lambda|x|}.
\]
For the first term on the r.h.s. of the above expression, we only need to decompose the summation into $\sum_{|y|<|x|/2}$ and $\sum_{|y|\geq |x|/2}$ to complete the proof.
\end{proof}

\begin{lem}
Let $u_\eps $ solve 
\[
(\alpha+\nabla_\eps^* \nabla_\eps)u_\eps=f_\eps
\]
on $\eps\Z^d$, with $f_\eps$ satisfying (i) there exists $\lambda>0$ s.t. $|f_\eps(x)|\les e^{-\lambda |x|}$, (ii) there exists a continuous function $\bar{f}:\R^d\to \R$ such that $\|f_\eps-\bar{f}\|_{2,\eps}\to0$ as $\eps\to 0$, then for $\bar{u}$ solving 
\[
(\alpha-\Delta)\bar{u}=\bar{f}
\]
on $\R^d$, we have $\|u_\eps-\bar{u}\|_{2,\eps}\to0$ as $\eps\to 0$.
\label{l:fide}
\end{lem}

\begin{proof}
We first note that $|\bar{f}(x)|\les e^{-\lambda|x|}$ for $x\in\R^d$. The rest of the proof is decomposed into three steps.

\emph{Step 1.} We write 
\[
u_\eps(x)=\sum_{y\in\eps\Z^d} \eps^2 \G^{\eps^2\alpha}(\frac{x}{\eps},\frac{y}{\eps})f_\eps(y),
\]
and define
\[
\bar{u}_\eps(x):=\sum_{y\in\eps\Z^d} \eps^2 \G^{\eps^2\alpha}(\frac{x}{\eps},\frac{y}{\eps})\bar{f}(y)
\]
on $\eps\Z^d$. The goal is to show $\|u_\eps-\bar{u}_\eps\|_{2,\eps}\to 0$ as $\eps\to 0$. By \eqref{eq:gres}, we deduce
\[
\begin{aligned}
|u_\eps(x)-\bar{u}_\eps(x)|\les& \sum_{y\in\eps\Z^d} \eps^2 \frac{e^{-\rho|x-y|}}{|\frac{x-y}{\eps}|_*^{d-2}}|f_\eps(y)-\bar{f}(y)|\\
\les&\eps^d \sum_{y\in\eps\Z^d,y\neq x}  \frac{e^{-\rho|x-y|}}{|x-y|^{d-2}}|f_\eps(y)-\bar{f}(y)|+\eps^2|f_\eps(x)-\bar{f}(x)|,
\end{aligned}
\]
thus
\[
\begin{aligned}
&\eps^d\sum_{x\in\eps\Z^d} |u_\eps(x)-\bar{u}_\eps(x)|^2 \\
\les &\eps^d \sum_{x\in\eps\Z^d} \left( \sum_{y\in\eps\Z^d,y\neq x} \eps^d\frac{e^{-\rho|x-y|}}{|x-y|^{d-2}} \sum_{y\in\eps\Z^d,y\neq x}\eps^d\frac{e^{-\rho|x-y|}}{|x-y|^{d-2}}|f_\eps(y)-\bar{f}(y)|^2\right)\\
+&\eps^d\eps^4\sum_{x\in\eps\Z^d}|f_\eps(x)-\bar{f}(x)|^2:=I_1+I_2
\end{aligned}
\]
It is clear by assumption that $I_2\to 0$, and  
\[
\begin{aligned}
I_1\les& \eps^d \sum_{x\in\eps\Z^d}\sum_{y\in\eps\Z^d,y\neq x}\eps^d\frac{e^{-\rho|x-y|}}{|x-y|^{d-2}}|f_\eps(y)-\bar{f}(y)|^2\\
\les &\eps^d \sum_{y\in\eps\Z^d}|f_\eps(y)-\bar{f}(y)|^2\to 0
\end{aligned}
\]
as $\eps\to 0$.

\emph{Step 2.} Define 
\[
\tilde{u}_\eps(x)=\eps^d\sum_{y\in\eps\Z^d,y\neq x} \cG^\alpha(x,y)\bar{f}(y)
\]
on $\eps\Z^d$, with $\cG^\alpha$ the continuous Green's function of $\alpha-\Delta$ in $\R^d$. The goal is to show $\|\bar{u}_\eps-\tilde{u}_\eps\|_{2,\eps}\to 0$ 
as $\eps\to 0$. We first have
\[
|\bar{u}_\eps(x)-\tilde{u}_\eps(x)|\les \sum_{y\in\eps\Z^d,y\neq x} |\eps^2 \G^{\eps^2\alpha}(\frac{x}{\eps},\frac{y}{\eps})-\eps^2\cG^{\eps^2\alpha}(\frac{x}{\eps},\frac{y}{\eps})||\bar{f}(y)|+\eps^2|\bar{f}(x)|.
\]
where we used the scaling property $\cG^{\eps^2\alpha}(x/\eps,y/\eps)=\eps^{d-2}\cG^\alpha(x,y)$. By \cite[Lemma 3.1]{conlon2}, 
\[
|\G^\lambda(x,y)-\cG^\lambda(x,y)|\les \frac{e^{-c \sqrt{\lambda}|x-y|}}{|x-y|_*^{d-1}}
\]
 for some $c>0$, so
\[
|\bar{u}_\eps(x)-\tilde{u}_\eps(x)|\les \sum_{y\in\eps\Z^d,y\neq x} \eps^{d+1} \frac{e^{-c\sqrt{\alpha}|x-y|}}{|x-y|^{d-1}}|\bar{f}(y)|+\eps^2|\bar{f}(x)| \les \eps e^{-c|x|}
\]
for some $c>0$, which implies $\|\bar{u}_\eps-\tilde{u}_\eps\|_{2,\eps}\to 0$.

\emph{Step 3.} We first extend $\tilde{u}_\eps(x)$ from $\eps\Z^d$ to $\R^d$ such that $\tilde{u}_\eps(x)=\tilde{u}_\eps([x]_\eps)$ with $[x]_\eps$ the $\eps-$integer part of $x$. Then we consider 
\[
\|\tilde{u}_\eps-\bar{u}\|_{2,\eps}^2=\eps^d\sum_{x\in\eps\Z^d} |\tilde{u}_\eps(x)-\bar{u}(x)|^2 =\int_{\R^d} |\tilde{u}_\eps(x)-\bar{u}([x]_\eps)|^2dx.
\]
It is clear that $\int_{\R^d} |\bar{u}([x]_\eps)-\bar{u}(x)|^2dx\to 0$ as $\eps\to 0$, so we only need to show 
\[
\int_{\R^d} |\tilde{u}_\eps(x)-\bar{u}(x)|^2dx\to 0.
\]
Since $\bar{u}(x)=\int_{\R^d} \cG^\alpha(x,y)\bar{f}(y)dy$, we have $\tilde{u}_\eps(x)\to \bar{u}(x)$ for $x\in\R^d$. Now by Lemma~\ref{l:de}, $|\tilde{u}_\eps(x)|+|\bar{u}(x)|\les e^{-c|x|}$ for some $c>0$, so by dominated convergence theorem, the proof is complete.
\end{proof}

\subsection*{Acknowledgments} We would like to thank the anonymous referee for several helpful suggestions. The author's research is partially funded by grant DMS-1613301 from the US National Science Foundation.

\def\cprime{$'$}


\end{document}